\documentclass[11pt]{amsart} 
\usepackage{amsmath} 
\usepackage{amssymb} 
\usepackage{amscd}
\usepackage{amsxtra}
\usepackage[hypertex]{hyperref}
\usepackage[margin=3cm]{geometry} 

\newtheorem{theorem}{Theorem}[section]
\newtheorem{lemma}[theorem]{Lemma}
\newtheorem{proposition}[theorem]{Proposition}
\newtheorem{proposition-definition}[theorem]{Proposition-Definition} 

\newtheorem{corollary}[theorem]{Corollary}
 
\newtheorem{claim}[theorem]{Claim} 
\newtheorem{sublemma}[theorem]{Sublemma} 

\theoremstyle{definition} 
\newtheorem{definition}[theorem]{Definition}
\newtheorem{remark}[theorem]{Remark}

\newtheorem{notation}[theorem]{Notation}

\newcommand{\C}{\mathbb C}

\newcommand{\Z}{\mathbb Z}
\newcommand{\Q}{\mathbb Q}
 
\newcommand{\Proj}{\mathbb P}

\newcommand{\ch}{\operatorname{ch}}
\newcommand{\Td}{\operatorname{Td}} 
\newcommand{\Aut}{\operatorname{Aut}} 
\newcommand{\Fl}{\operatorname{Fl}}
\newcommand{\Gr}{\operatorname{Gr}}

\newcommand{\hP}{\widehat{P}}

\newcommand{\Eff}{\operatorname{Eff}}
\newcommand{\ev}{\operatorname{ev}}

\newcommand{\id}{\operatorname{id}}
\newcommand{\End}{\operatorname{End}}
\newcommand{\Ad}{\operatorname{Ad}} 
\newcommand{\cA}{\mathcal{A}} 
\newcommand{\cO}{\mathcal{O}}
 
\newcommand{\cF}{\mathcal{F}}

\newcommand{\tB}{\widetilde{B}} 
 
\newcommand{\tT}{\widetilde{T}} 
\newcommand{\tS}{\widetilde{S}} 
\newcommand{\tJ}{\widetilde{J}} 

\newcommand{\la}{\left\langle}
\newcommand{\ra}{\right\rangle}

\newcommand{\ov}{\overline}

\def\pair#1#2{\langle #1,#2\rangle}
\def\parfrac#1#2{\frac{\partial{#1}}{\partial #2}}
\def\corr#1{\left\langle #1 \right\rangle} 
\def\Corr#1{\left\langle\!\left\langle #1 \right \rangle\!\right\rangle} 
\def\BigCorr#1{\left\langle\!\!\!\left\langle #1 \right \rangle \!\!\!\right\rangle}

\begin{document}
\title[Reconstruction and Convergence of Quantum $K$-Theory]
{Reconstruction and Convergence in Quantum $K$-Theory 
via Difference Equations}

\author{Hiroshi Iritani} 
\address{Department of Mathematics, Graduate School of Science, 
Kyoto University, Kitashirakawa-Oiwake-cho, Sakyo-ku, Japan} 
\email{iritani@math.kyoto-u.ac.jp}

\author{Todor Milanov} 
\address{Kavli Institute for the Physics and Mathematics of the 
Universe (WPI), University of Tokyo, 
Kashiwa, Chiba 277-8583, Japan} 
\email{todor.milanov@ipmu.jp} 

\author{Valentin Tonita}
\address{Kavli Institute for the Physics and Mathematics of the 
Universe (WPI), University of Tokyo, 
Kashiwa, Chiba 277-8583, Japan} 
\email{valentin.tonita@ipmu.jp} 

\begin{abstract} 
We give a new reconstruction method of big quantum $K$-ring 
based on the $q$-difference module structure in quantum $K$-theory 
\cite{Givental-Lee, Givental-Tonita}. 
The $q$-difference structure yields commuting linear operators  
$A_{i,\rm com}$ on the $K$-group as many as the Picard number 
of the target manifold. 
The genus-zero quantum $K$-theory can be reconstructed 
from the $q$-difference structure at the origin $t=0$ if the $K$-group 
is generated by a single element under the actions of $A_{i,\rm com}$. 
This method allows us to prove the convergence of 
the big quantum $K$-rings of certain manifolds, including 
the projective spaces and the complete flag manifold $\Fl_3$. 
\end{abstract} 

\maketitle

\section{Introduction} 
The $K$-theoretic Gromov-Witten (GW for short) invariants 
have been introduced by Givental and Y.~P.~Lee \cite{Givental:WDVVK, Givental-Lee, 
Lee:foundation} for a smooth projective variety $X$. 
For vector bundles $E_1,\dots, E_n$ on $X$ and non-negative 
integers $k_1,\dots,k_n$, the $K$-theoretic GW invariant 
\[
\corr{E_1 L^{k_1},\dots, E_n L^{k_n}}_{g,n,d}^X \in \Z 
\]
is defined as the Euler characteristic of 
$\cO^{\rm vir} \otimes (\bigotimes_{i=1}^n \ev_i^*(E_i) \otimes L_i^{k_i})$ 
over the moduli space of 
genus-$g$ stable maps to $X$ of degree $d$ and with 
$n$ markings. Here $\cO^{\rm vir}$ is the virtual 
structure sheaf \cite{Lee:foundation} and $L_i$ is the universal 
cotangent line bundle at the $i$-th marked point. 
As in the case of cohomological GW theory, 
the genus-zero $K$-theoretic GW invariants define the quantum $K$-ring, 
a deformation of the topological $K$-ring $K(X)$. 
The quantum $K$-ring has two kinds of 
deformation parameters --- Novikov parameters $Q_1,\dots,Q_r$ 
which correspond to a nef basis of $H^2(X,\Z)$ and 
parameters $t^0,\dots, t^N$ which 
correspond to a basis $\Phi_0,\dots,\Phi_N$ of $K(X)$. 
Along the $t$-parameter deformation, the quantum $K$-ring 
forms a holonomic differential system.
Namely we have a flat connection on the trivial 
$K(X)$-bundle over $K(X)$: 
\[
\nabla^q_\alpha = (1-q) \parfrac{}{t^\alpha} + (\Phi_\alpha \bullet),  
\qquad 0\le \alpha \le N 
\]
where $(\Phi_\alpha\bullet)$ denotes the quantum multiplication 
by $\Phi_\alpha\in K(X)$ and $q$ is a formal 
parameter. 
This defines a Frobenius-type structure on the (formal neighborhood 
of the origin in) $K(X)$ \cite{Givental:WDVVK, Lee:foundation} 
which is analogous to the Frobenius 
manifold of quantum cohomology \cite{Dubrovin:2DTFT}.  

On the other hand, the behavior of the quantum $K$-ring with 
respect to Novikov parameters $Q_i$ has been more mysterious 
due to the absence of divisor equation. 
Givental-Lee \cite{Givental-Lee} showed that the small (i.e.~$t=0$) 
quantum $K$-ring of a type $A$ flag manifold is governed by $q$-difference 
equations, which turns out to be the difference quantum Toda lattice. 
A relation to the Toda lattice has been studied further 
by Braverman-Finkelberg 
\cite{Braverman-Finkelberg:type_A, Braverman-Finkelberg:othertypes}. 
In a recent work \cite{Givental-Tonita}, 
Givental and Tonita 
characterized the $K$-theoretic GW theory at genus zero 
in terms of cohomological GW theory and 
thereby showed the existence of a difference module structure 
in quantum $K$-theory for a general target.  
Using their result, we introduce a \emph{$q$-shift operator} 
\[
\cA_i = A_i 
q^{Q_i\partial_{Q_i}}, \quad 
A_i \in \End(K(X))\otimes \Q[q,q^{-1}][\![Q,t]\!]
\]
acting on the $K$-group, which yields the shift $Q_i \mapsto q Q_i$ 
of the Novikov variable $Q_i$. 
A geometric meaning of the operator $\cA_i$ is given 
in terms of graph spaces (Proposition \ref{prop:graphinvariants}). 
The $q$-shift operators $\cA_i$ commute with the above flat connection 
$\nabla^q_\alpha$, and therefore satisfy the following 
Lax equation: 
\[
(1-q) \parfrac{\cA_i}{t^\alpha} = [\cA_i, (\Phi_\alpha\bullet)]. 
\]
By setting $q=1$, we obtain endomorphisms 
$A_{i,\rm com} = \cA_i|_{q=1} \in \End(K(X)) \otimes \Q[\![Q,t]\!]$ 
commuting with quantum multiplications: 
\[
[A_{i,\rm com}, (\Phi_\alpha\bullet)] =0. 
\]
In particular $A_{i,\rm com}$ represents the quantum multiplication 
by $A_{i,\rm com}1\in K(X)\otimes \Q[\![Q,t]\!]$. 
These commuting operators sometimes generate the quantum $K$-ring; 
in such circumstances, we show that the genus-zero big quantum $K$-theory 
can be reconstructed from $\cA_i|_{t=0}$ via the Lax equation. 
\begin{theorem}[Theorem \ref{thm:reconstruction}, 
Lemma \ref{lem:rational_expression}]
\label{thm:introd_reconstruction} 
Suppose that $K(X)$ is generated by a single 
element by the iterated action of the operators 
$A_{i,\rm com}|_{t=0}$, $i=1,\dots,r$ over 
the fraction field of the Novikov ring $\Q[\![Q]\!]$. 
Then: 

(1) the $q$-shift operators $\cA_i$ can be reconstructed 
from its restriction $\cA_i|_{t=0}$ to $t=0$; 

(2) the quantum products $(\Phi_\alpha\bullet)$ 
have rational function expressions 
in the entries of the operators $A_{i,\rm com}$, 
where the rational functions are determined only by $A_{i,\rm com}|_{t=0}$. 
\end{theorem} 

The assumption in the theorem is satisfied if the classical $K$-ring 
$K(X)$ is generated by line bundles. 
In this case, the small $q$-shift operators $\cA_i|_{t=0}$ 
can be computed from the so-called \emph{small $J$-function} $J|_{t=0}$ 
(Lemma \ref{lem:initial_condition}) that encodes 
genus-zero one-point descendant invariants. 
In particular our theorem recovers the reconstruction theorem 
of Lee-Pandharipande \cite{Lee-Pandharipande} 
(Corollary \ref{cor:Lee-Pandharipande}). 
Note however that our method is completely different from theirs. 
It is also more effective, in that it gives
the whole degree $d$ potential (recursively on $d$) at once 
rather than single invariants.
In \S \ref{sec:example}, we shall compute explicitly $\C\Proj^1$ invariants 
up to degree $4$ and $\C\Proj^2$ up to degree $3$ 
in our method. 

Another application of our method is to recover and 
provide new examples of finiteness results of the small quantum product 
i.e.~the structure constants vanish for large degrees. 
This is known to hold for a certain class of homogeneous spaces 
of Picard rank one \cite{Buch-Mihalcea, BCMP:cominuscule, BCMP}. 
In \S \ref{sec:example}, we confirm finiteness 
for the projective space $\C\Proj^N$ and 
the complete flag manifold $\Fl_3$ by writing explicitly 
the small quantum multiplication tables. 
To the best of our knowledge the results for $\Fl_3$ 
were not known before (see \cite{Lenart-Maeno} 
for a relevant conjecture).

In cohomological GW theory, the analyticity of quantum cohomology 
plays a crucial role in many applications, 
especially to mirror symmetry and integrable systems. 
We recall, for example, theory of semisimple Frobenius manifolds 
and higher-genus reconstruction \cite{Dubrovin:2DTFT, 
Givental:quantization, Teleman:classification}.  
The analyticity of quantum $K$-ring should be  therefore 
important when one pursues analogous stories in quantum $K$-theory. 
It turns out that the differential-difference system that 
we used in our reconstruction theorem can be used also 
to address the convergence properties of the big quantum $K$-product. 
First, we prove analyticity in a weaker sense, which however holds 
for all target manifolds $X$. 
\begin{theorem}[Theorem \ref{thm:polynomiality}] 
\label{thm:polynomiality_introd}
Choose a basis $\Phi_0,\dots, \Phi_N$ of $K(X)$ so that 
\begin{itemize} 
\item[(i)] $\Phi_0 = 1$; 
\item[(ii)] $\ch(\Phi_1),\dots,\ch(\Phi_r)$ belong to $H^{\ge 2}(X)$ 
and $\ch_1(\Phi_1),\dots,\ch_1(\Phi_r)$ form a 
nef integral basis of $H^2(X;\Z)/{\rm torsion}$; 
\item[(iii)] $\ch(\Phi_{r+1}),\dots,\ch(\Phi_N)$ belong to  
$H^{\ge 4}(X)$. 
\end{itemize} 
Fix classes $E_1,\dots, E_n \in K(X)$, non-negative integers 
$k_1,\dots,k_n\ge 0$ and a degree $d\in \Eff(X)$. 
Then the generating function 
\[
\sum_{m\ge 0}
\frac{1}{m!} 
\corr{E_1L^{k_1},\dots,E_n L^{k_n}, t,\dots,t}_{0,m+n, d}^X 
\]
with $t = \sum_{\alpha=0}^N t^\alpha \Phi_\alpha$ is a 
polynomial in $t^0,\dots,t^N$ and $e^{t^0}, e^{t^1},\dots, e^{t^r}$. 
In particular, the big quantum 
$K$-product is defined over $\Q[t^1,\dots,t^N, e^{t^1},\dots,e^{t^r}][\![Q]\!]$. 
\end{theorem}

Finally we prove the convergence of the big quantum $K$-theory 
under certain restrictive assumptions on the manifold $X$: 

\begin{theorem}[Theorem \ref{thm:convergence}]
\label{thm:convergence_introd}
Suppose that $K(X)$ is generated by line bundles as a ring. 
If the $q$-shift operators  $\cA_i|_{t=0}$ 
at $t=0$ are convergent, 
then the big quantum products $(\Phi_\alpha\bullet)$ 
and the $q$-shift operators $\cA_i$ reconstructed via 
Theorem \ref{thm:introd_reconstruction} are convergent. 
\end{theorem} 

It follows that the big quantum $K$-rings of $\C\Proj^N$
and $\Fl_3$ are analytic and generically semisimple  
(Propositions \ref{prop:projective}, \ref{prop:flag3}). 
It will be interesting to find out if these quantum $K$-rings 
admit Landau-Ginzburg ``mirror'' descriptions. 
This question is one of the motivations for our work 
and Theorem \ref{thm:convergence_introd} may be viewed as 
the first step towards our goal.

\section{Quantum $K$-Theory and Difference Module Structure}

\subsection{Notation} 
We work with a smooth projective variety and its topological $K$-group 
of degree zero, using the following notation: 
\vspace{5pt} 
\begin{center} 
\begin{tabular}{ll}
$X$ & a smooth projective variety over $\C$; \\ 
$\Eff(X)$ & the semigroup of classes of 
effective curves in $H_2(X;\Z)$; \\ 
$K(X)=K^0(X;\Q)$ & the Grothendieck group of topological 
complex vector bundles \\ 
& on $X$ with rational coefficients; \\ 
$\Phi_0,\dots,\Phi_N$ & a basis of $K(X)$ such that 
$\Phi_0$ is the identity class $1=[\cO_X]$; \\ 
$t^0,\dots,t^N$ & linear co-ordinates on $K(X)$ 
dual to $\Phi_0,\dots,\Phi_N$; \\
$t = \sum_{\alpha=0}^N t^\alpha \Phi_\alpha$ & a general point in $K(X)$; \\ 
$P_1,\dots,P_r$ & classes of line bundles such that 
$p_1 := c_1(P_1),\dots, p_r := c_1(P_r)$ \\ 
& form a nef integral basis of 
$H^2(X;\Z)/{\rm torsion}$; \\ 
$Q_1,\dots,Q_r$ & Novikov variable dual to $P_1,\dots,P_r$; \\
& we write $Q^d = \prod_{i=1}^r Q_i^{d_i}$ with 
$d_i = \pair{p_i}{d}$ for $d\in H_2(X;\Z)$; \\
$g(E,F)$ & the $K$-theory pairing $\chi(E\otimes G)$ on $K(X)$; \\
& we write $g_{\alpha\beta} = g(\Phi_\alpha, \Phi_\beta)$; \\ 
\end{tabular} 
\end{center} 

Because of our choices of $p_1,\dots,p_r$, $Q^d$ does not 
contain negative powers of $Q_1,\dots,Q_r$ 
when $d$ is an effective class. 
We write $\Q[\![Q]\!] := \Q[\![Q_1,\dots, Q_r]\!]$  
for the Novikov ring and $\Q(\!(Q)\!)$ for the ring of fractions 
of $\Q[\![Q]\!]$. 

\subsection{$K$-theoretic GW invariants} 
Givental and Y.~P.~Lee \cite{Givental:WDVVK, Givental-Lee, Lee:foundation} 
introduced $K$-theoretic GW invariants for 
smooth projective varieties. 
We concentrate on the genus-zero part of the theory. 
Let $X_{0,n,d}$ denote the moduli stack of stable maps to $X$ of 
genus-zero, degree $d$ and with $n$ marked points. 
In place of the virtual fundamental cycle in the cohomological 
GW theory, we use the \emph{virtual structure sheaf} 
$\cO^{\rm vir} \in K_0(X_{0,n,d})$ on $X_{0,n,d}$ 
\cite{Lee:foundation}. 
Let $L_i$ denote the universal cotangent line bundle 
on $X_{0,n,d}$ at the $i$-th marked point 
and let $\ev_i \colon X_{0,n,d} \to X$ denote the evaluation 
map. 
For $E_1,\dots, E_n \in K(X)$ and non-negative integers 
$k_1,\dots,k_n$, the $K$-theoretic 
GW invariant is defined to be: 
\[
\corr{E_1 L^{k_1}, E_2 L^{k_2}, \cdots, E_n L^{k_n}}_{0,n,d}^X 
=\chi\left( 
\cO^{\rm vir} \otimes \bigotimes_{i=1}^n L_i^{k_i}\ev_i^*(E_i)  
\right). 
\]
When each class $E_i$ is integral, the $K$-theoretic GW invariant is 
an integer. 

\subsection{Big quantum $K$-ring} 
The genus-zero $K$-theoretic GW potential 
$\cF=\cF(t)$ of $X$ is defined to be: 
\[
\cF= \sum_{d \in \Eff(X)} \sum_{n\ge 0} 
\corr{t,\dots,t}_{0,n,d}^X \frac{Q^d}{n!} 
\]
where $t = \sum_{\alpha=0}^N t^\alpha \Phi_\alpha$ 
is a general point in $K(X)$.  
The potential $\cF$ is an element of $
\Q[\![Q,t]\!] = \Q[Q_1,\dots,Q_r, t^0,\dots,t^N]\!]$. 
We introduce the (non-constant) pairing $G$ as: 
\[
G(\Phi_\alpha, \Phi_\beta) = G_{\alpha\beta} 
=  \partial_0 \partial_\alpha \partial_\beta \cF. 
\]
where $\partial_\alpha = \partial/\partial t^\alpha$. 
The pairing $G$ is extended to a $\Q[\![Q,t]\!]$-bilinear pairing 
on $K(X) \otimes \Q[\![Q,t]\!]$. 
We have $G_{\alpha\beta} = g_{\alpha\beta} + 
\partial_\alpha \partial_\beta \cF$ 
and $G_{\alpha\beta}|_{Q=t=0} = g_{\alpha\beta}$. 
The quantum product $\Phi_\alpha\bullet \Phi_\beta$ of basis elements 
is defined by 
\[
G(\Phi_\alpha \bullet \Phi_\beta, \Phi_\gamma) 
= \partial_\alpha \partial_\beta \partial_\gamma \cF.  
\]
The quantum product $\bullet$ is extended bilinearly over $\Q[\![Q,t]\!]$ 
and defines the \emph{big quantum $K$-ring} 
$(K(X) \otimes \Q[\![Q,t]\!], \bullet)$. 
This has the following properties: 
\begin{enumerate} 
\item the quantum product is associative and commutative; 
the associativity is due to the $K$-theoretic WDVV equation 
\cite{Givental:WDVVK, Lee:foundation}; 
\item $\Phi_0=1$ gives the identity of 
the quantum product; 
\item the classical limit ($Q\to 0$) of the quantum product 
is the tensor product:  
$\lim_{Q\to 0} \Phi_\alpha \bullet \Phi_\beta 
= \Phi_\alpha \otimes \Phi_\beta$; 
\item it is a Frobenius algebra: 
$G(\Phi_\alpha\bullet \Phi_\beta, \Phi_\gamma) = 
G(\Phi_\alpha, \Phi_\beta \bullet \Phi_\gamma)$. 
\end{enumerate}  
The restriction of the big quantum $K$-ring to $t=0$ is 
still a non-trivial family of rings parametrized by $Q_1,\dots,Q_r$. 
We call the restriction to $t=0$ the \emph{small} quantum $K$-ring 
as opposed to the big quantum $K$-ring. 

\subsection{Quantum connection and fundamental solution} 
Let $q$ be a formal variable. 
The quantum connection is the operator 
\[
\nabla^q_\alpha  = (1-q) \parfrac{}{t^\alpha} +  
\Phi_\alpha \bullet \qquad 0\le \alpha \le N 
\]
acting on $K(X)\otimes \Q[q,q^{-1}][\![Q,t]\!]$. 
This may be viewed as a connection on the tangent bundle 
$TK(X)$ by identifying $\Phi_\alpha$ with the co-ordinate 
vector field $\partial/\partial t^\alpha$. 
Givental \cite{Givental:WDVVK} showed that 
the quantum connection is flat, i.e.~$[\nabla^q_\alpha, 
\nabla^q_\beta] =0$ and constructed a fundamental 
solution. 
Introduce the generating function: 
\[
S_{\alpha \beta} = g_{\alpha\beta} 
+ \sum_{d\in \Eff(X), n\ge 0} 
\frac{Q^d}{n!} \corr{\Phi_\alpha, t, \dots, t, 
\frac{\Phi_\beta}{1-qL}}_{0,n+2,d}^X 
\]
where $\Phi_\beta/(1-qL)$ in the correlator 
should be expanded in the power series 
$\sum_{n=0}^\infty \Phi_\beta q^n L^n$. 
Applying Kawasaki-Riemann-Roch theorem on the moduli space 
$X_{0,n,d}$ (see \cite{Givental-Tonita}), 
we find that the correlator 
\[
\corr{\Phi_\alpha, t,\dots, t,\frac{\Phi_\beta}{1-qL}}_{0,n+2,d}
\] 
is a rational function in $q$ which has 
poles only at roots of unity and vanishes at $q=\infty$. 
If $X_{0,n,d}$ were a manifold, this has poles only 
at $q=1$ since $L-1$ is nilpotent. 
The orbifold singularities of $X_{0,n,d}$ yield poles at 
other roots of unity. 
We use the following notation: 
\begin{notation} 
For a rational function $f = f(q)$ in $q$, we write  
$\ov{f}=f(q^{-1})$. In particular, we write $\ov{q} = q^{-1}$. 
\end{notation} 
We introduce the endomorphism-valued functions 
$S, T \in \End(K(X))\otimes \Q(q) [\![Q,t]\!]$  
by the formula: 
\[
G(\Phi_\alpha, S \Phi_\beta) = \ov{S}_{\alpha\beta}, \qquad 
g(T \Phi_\alpha, \Phi_\beta) = S_{\alpha\beta}. 
\]
Note that $\ov{S}$ and $T$ are adjoint to each other 
with respect to the pairings $G$ and $g$. 
When we expand $S$ and $T$ as power series in $t$ and $Q$, 
each coefficient is a rational function in $q$ 
having poles only at roots of unity. 
Note that we have 
\[
S|_{q=\infty} = T|_{q=\infty} = \id, \qquad 
S|_{Q=t=0} = T|_{Q=t=0} = \id. 
\]

\begin{theorem}[Givental \cite{Givental:WDVVK}, 
Y.~P.~Lee \cite{Lee:foundation}] 
\label{thm:fund_sol} 
The endomorphism-valued functions  
$S$ and $T$ give fundamental solutions to 
the differential equations: 
\begin{align*} 
& (1-q)\partial_\alpha S+ \Phi_\alpha \bullet S =0 \\ 
& (1-q) \partial_\alpha T = T (\Phi_\alpha \bullet)
\end{align*} 
for $0\le \alpha \le N$. 
In other words, we have 
$\nabla^q_\alpha \circ S  = S \circ (1-q) \partial_\alpha$ 
and 
$T \circ \nabla^q_\alpha= (1-q) \partial_\alpha \circ T$. 
\end{theorem} 

\begin{proposition}
\label{prop:unitarity}
We have the following: 
\begin{enumerate}
\item $S = T^{-1}$.  

\item 
$g(\ov{T} \Phi_\alpha, T \Phi_\beta) = G(\Phi_\alpha,\Phi_\beta)$.

\item  
$G(\ov{S} \Phi_\alpha, S \Phi_\beta) = g(\Phi_\alpha, \Phi_\beta)$.  
\end{enumerate} 
\end{proposition} 
\begin{proof} 
By the definition of $S$ and $T$, statements (1)--(3) are all equivalent. 
It suffices to show statement (3). 
Writing $G^{\alpha\beta}$ for the coefficients of the inverse 
matrix of $(G_{\alpha\beta})$ and adopting Einstein's summation 
convention, we have $S \Phi_\beta = 
G^{\gamma\alpha} \ov{S}_{\alpha\beta} \Phi_\gamma$. 
The statement (3) can be rewritten as: 
\[
S_{\epsilon \alpha} G^{\epsilon\gamma} \ov{S}_{\gamma\beta} 
= g_{\alpha\beta}.  
\]
The $K$-theoretic WDVV equation says that 
\begin{align*} 
(\partial_0 S_{\epsilon\alpha})
 G^{\epsilon\gamma} (\partial_0\ov{S}_{\gamma \beta}) 
&  = \BigCorr{\Phi_\epsilon,\Phi_0,\frac{\Phi_\alpha}{1-q L}}_t
 G^{\epsilon \gamma} 
 \BigCorr{\Phi_\gamma, \Phi_0, \frac{\Phi_\beta}{1-\ov{q} L}}_t 
\end{align*} 
is symmetric with respect to the permutation of the insertions 
$\Phi_0$, $\Phi_\alpha/(1-qL)$, $\Phi_0$, $\Phi_\beta/(1-\ov{q}L)$. 
Here we used the notation: 
\[
\Corr{\alpha_1,\dots,\alpha_k}_t := 
\sum_{d,n} 
\corr{\alpha_1,\dots,\alpha_k,t,\dots,t}_{0,k+n, d}^X 
\frac{Q^d}{n!}.  
\]
Therefore we have 
\begin{align*} 
(\partial_0 S_{\epsilon\alpha})
G^{\epsilon\gamma} (\partial_0\ov{S}_{\gamma \beta}) 
&=\Corr{\Phi_0,\Phi_0, \Phi_\epsilon}_t G^{\epsilon\gamma} 
\BigCorr{\Phi_\gamma, 
\frac{\Phi_\alpha}{1-q L}, \frac{\Phi_\beta}{1-\ov{q} L} }_t \\
& 
= \BigCorr{\Phi_0, \frac{\Phi_\alpha}{1-q L}, 
\frac{\Phi_\beta}{1-\ov{q} L}}_t \\ 
& = \left(1+ \frac{q}{1-q} + 
\frac{\ov{q}}{1-\ov{q}} \right) 
\BigCorr{\frac{\Phi_\alpha}{1-q L}, 
\frac{\Phi_\beta}{1-\ov{q} L}}_t 
+ \frac{1}{(1-q)(1-\ov{q})} g_{\alpha\beta} \\ 
& = \frac{1}{(1-q)(1- \ov{q})} g_{\alpha\beta}.  
\end{align*} 
In the third line we used the string equation (see \cite[\S 4.4]{Lee:foundation}). 
Theorem \ref{thm:fund_sol} implies 
$(1-q) \partial_0 S_{\alpha\beta}= S_{\alpha\beta}$ 
(which also follows from the string equation) 
and therefore we have 
$S_{\epsilon \alpha}G^{\epsilon \gamma}\ov{S}_{\gamma \beta}  
= g_{\alpha\beta}$ as desired. 
\end{proof} 

\begin{definition} 
The \emph{$J$-function} is defined as 
\begin{align*} 
J(q,Q, t) &= (1-q) S^{-1} \Phi_0 = (1-q) T \Phi_0 \\ 
& =  (1-q)\left[ \Phi_0 + \sum_{d\in \Eff(X)} \sum_{n\ge 0} 
 \frac{Q^d}{n!} 
\corr{\Phi_0, t,\dots,t,\frac{\Phi_\alpha}{1-qL}}_{0,n+2,d}^X 
g^{\alpha\beta} \Phi_\beta\right] \\ 
& = (1-q) \Phi_0 + t + 
\sum_{d\in \Eff(X)} \sum_{n\ge 0} 
\frac{Q^d}{n!} 
\corr{t,\dots,t,\frac{\Phi_\alpha}{1-qL}}_{0,n+2,d}^X 
g^{\alpha\beta} \Phi_\beta. 
\end{align*} 
\end{definition} 

\begin{remark} 
The variable $q$ in the quantum connection 
should be viewed as a generator of the 
$S^1$-equivariant $K$-group of a point 
$K_{S^1}({\rm pt}) \cong \Q[q,q^{-1}]$. 
Recall that the cohomological quantum 
connection is given by a similar formula:
$\nabla^z_\alpha = z \partial_\alpha  + 
\phi_\alpha \bullet$, where 
$\{\phi_\alpha \}$ is a basis of $H(X)$. 
The variable $z$ here corresponds to a generator 
of $H_{S^1}({\rm pt}) = \Q[z]$ 
and is related to $q$ by $q = e^{-z}$. 
See \cite{Givental-Lee} and \S \ref{subsec:graph}.
\end{remark} 

\subsection{Difference Module Structure} 
Givental-Lee \cite{Givental-Lee} observed that the quantum $K$-theory of 
a type A flag manifold has the structure of a difference module. 
They moreover identified it with the difference Toda lattice. 
Givental-Tonita \cite{Givental-Tonita} showed the existence of 
a difference module structure for a general target. 
The difference module structure may be regarded 
of as a replacement for the divisor equation in quantum $K$-theory.  

Givental-Tonita showed that the space 
\[
S^{-1} \left( K(X) \otimes \Q[q,q^{-1}][\![Q,t]\!] \right), 
\]
which appears as a tangent space to Givental's Lagrangian cone, 
is preserved by the operator $P_i^{-1} q^{Q_i\partial_{Q_i}}$ 
for all $1\le i\le r$. 
Here $q^{Q_i\partial_{Q_i}}$ is the $q$-shift operator acting 
on functions in $Q_1,\dots,Q_r$ as 
\[
f(Q_1,\dots,Q_r) \longmapsto f(Q_1,\dots,Q_{i-1}, q Q_i, Q_{i+1}, 
\dots, Q_r). 
\]
Therefore \emph{the endomorphism $A_i$ defined by}  
\begin{equation}
\label{eq:def_A}
A_i = S \left( P_i^{-1} q^{Q_i\partial_{Q_i}} S^{-1} \right) 
= T^{-1}\left( P_i^{-1} q^{Q_i\partial_{Q_i}} T \right) 
\end{equation} 
\emph{lies in} $\End(K(X)) \otimes \Q[q,q^{-1}][\![Q,t]\!]$. 
What is crucial here is the fact that $A_i$ is defined over 
$\Q[q,q^{-1}]$ and is regular at $q=\text{roots of unity}$  
whereas $S$ and $T$  
have poles only at roots of unity (as functions in $q$). 
Setting $\cA_i = A_i q^{Q_i\partial_{Q_i}}$, we have 
\begin{equation} 
\label{eq:A_S}
S \circ P_i^{-1} q^{Q_i\partial_{Q_i}} = \cA_i \circ S, 
\qquad 
P_i^{-1} q^{Q_i\partial_{Q_i}} \circ T = T \circ \cA_i. 
\end{equation} 
We call $\cA_i$ \emph{the $q$-shift operator}. 
We can regard $S$ (or $T$) as a fundamental solution 
to $q$-difference equations in the Novikov variables. 
The existence of a fundamental solution implies 
the compatibility between the difference equation in $Q$ 
and differential equation in $t$. 
This difference-differential system 
is compatible with the pairing $G$ by Proposition \ref{prop:unitarity}. 
\begin{proposition} 
\label{prop:compatibility}
The operators $\nabla^q_\alpha$, $\cA_i = A_i q^{Q_i\partial_{Q_i}}$ 
and the pairing $G$ satisfy the following compatibility equations: 
\begin{enumerate} 
\item $[\nabla^q_\alpha,\nabla^q_\beta] = 
[\cA_i, \cA_j] = [\nabla^q_\alpha, \cA_i]=0$ for all 
$0\le \alpha,\beta \le N$ and $1\le i, j\le r$; 

\item for $s_1, s_2 \in K(X) \otimes \Q[q,q^{-1}][\![Q,t]\!]$ and 
$0\le \alpha \le N$, $1\le i\le r$, we have 
\begin{align*} 
\partial_\alpha G(\ov{s_1}, s_2) & = 
G\left(\ov{\tfrac{1}{1-q} \nabla_\alpha^q s_1}, s_2\right) + 
G\left(\ov{s_1}, \tfrac{1}{1-q}\nabla^q_\alpha  s_2\right),  
\\
q^{Q_i \partial_{Q_i}} G(\ov{s_1}, s_2) 
& = G\left(\ov{\cA_i^{-1} s_1}, \cA_i s_2\right).
\end{align*} 
\end{enumerate}
\end{proposition} 
\begin{proof} 
Part (1) follows from the existence of a fundamental solution $S$ 
satisfying 
\[
\nabla^q_\alpha \circ S = S\circ (1-q)\partial_\alpha, 
\qquad 
\cA_i \circ S = S \circ P_i^{-1} q^{Q_i\partial_{Q_i}}  
\]
and the fact that $(1-q)\partial_\alpha$ and 
$P_i^{-1} q^{Q_i\partial_{Q_i}}$ commute each other. 
Part (2) follows from Proposition \ref{prop:unitarity}. 
We have 
\begin{align*} 
\partial_\alpha G(\ov{s_1},s_2) 
& = \partial_\alpha g( \ov{T s_1}, T s_2) 
= g(  \ov{\partial_\alpha T s_1}, T s_2) + g(\ov{T s_1}, \partial_\alpha T s_2) \\ 
& = g\left(\ov{T \tfrac{1}{1-q}\nabla^q_{\alpha}s_1}, Ts_2\right) 
+ g \left( \ov{T s_1}, T \tfrac{1}{1-q}\nabla^q_\alpha s_2\right) \\
& = G\left(\ov{\tfrac{1}{1-q}\nabla^q_\alpha s_1}, s_2\right) 
+ G\left( \ov{s_1}, \tfrac{1}{1-q}\nabla^q_\alpha s_2\right).  
\end{align*} 
We also have 
\begin{align*} 
q^{Q_i\partial_{Q_i}} G(\ov{s_1}, s_2)  
& = q^{Q_i\partial_{Q_i}} g( \ov{Ts_1}, Ts_2 ) 
= g\left( 
\ov{P_i q^{-Q_i\partial_{Q_i}} Ts_1}, P_i^{-1} q^{Q_i\partial_{Q_i}}T s_2
\right) \\
& = g\left( \ov{T \cA_i^{-1} s_1}, T \cA_i s_2\right)  
= G\left(\ov{\cA_i^{-1} s_1}, \cA_i s_2\right). 
\end{align*} 
The proposition is proved. 
\end{proof} 

\begin{remark} 
Because $q=-1$ is a fixed point of the involution 
$q \mapsto \ov{q} = q^{-1}$, the connection 
$\frac{1}{1-q}\nabla^{q}\big|_{q=-1}$ 
preserves the pairing $G$. Since $\frac{1}{1-q}\nabla^q$ is torsion-free 
as a connection on $TK(X)$, 
it follows that $\frac{1}{1-q}\nabla^q\big|_{q=-1}$ is the Levi-Civita 
connection of $G$. It follows that $G$ gives a flat metric 
on $T K(X)$ as discussed in \cite{Givental:WDVVK}. 
\end{remark} 
\begin{remark} 
\label{rem:Lax} 
The above compatibility equations imply 
\begin{align*} 
(1-q) \partial_\alpha A_i &= A_i \left( 
q^{Q_i\partial_{Q_i}}\Omega_\alpha \right) 
- \Omega_\alpha  A_i \\
 G(\Phi_\alpha, A_i\Phi_\beta) 
& = q^{Q_i\partial_{Q_i}}G(\ov{A_i}\Phi_\alpha, \Phi_\beta) 
\end{align*} 
where we write $\Omega_\alpha$ for the quantum multiplication 
$\Phi_\alpha\bullet$. 
\end{remark} 

\begin{corollary}
\label{cor:Acom} 
Set $A_{i,\rm com} = A_i|_{q=1} \in \End(K(X))\otimes \Q[\![Q,t]\!]$. 
Then $A_{i,\rm com}$, $i=1,\dots, r$ commute with the quantum 
multiplication $\Phi_\alpha \bullet$, $\alpha = 0,\dots, N$. 
\end{corollary} 

\begin{proposition} 
\label{prop:A_expansion} 
The endomorphism $A_i\in\End(K(X))\otimes \Q[q,q^{-1}][\![Q,t]\!]$ 
has the Taylor expansion of the form: 
\begin{equation*} 
A_i = P_i^{-1} + 
\sum_{\substack{d\in \Eff(X) \\ d_i > 0}} \sum_{k=0}^{d_i-1} 
A_{i,d,k}(t) (1-q)^k Q^d 
\end{equation*} 
where $d_i = \pair{p_i}{d}$ and 
$A_{i,d,k}(t) \in \Q[\![t^0,\dots,t^N]\!]$ is independent of $t^0$. 
\end{proposition} 
\begin{proof} 
By the compatibility equation in Proposition \ref{prop:compatibility},  
we have $(1-q)\partial_0 \cA_i = [\cA_i, (\Phi_0\bullet)] =0$ 
and thus $A_i$ is independent of $t^0$. 
Set $A_i = \sum_{d} A_{i,d} Q^d$, 
$S = \sum_{d} S_d Q^d$, $S^{-1} = T =\sum_{d} T_d Q^d $. 
It suffices to show that $A_{i,d} = P_i^{-1} \delta_{d,0}$ 
if $d_i = 0$ and 
$A_{i,d}$ is a polynomial of degree $d_i-1$ in $q$ if $d_i>0$. 
By the definition \eqref{eq:def_A}, we have 
\[
A_{i,d} = \sum_{d= d' + d''} S_{d'} P_i^{-1} 
q^{d''_i} T_{d''}. 
\]
Recall that $S_d$ and $T_d$ are regular at both $q=0$ and $q=\infty$; 
moreover $S_d$ and $T_d$ vanish at $q=\infty$ for $d\neq 0$. 
Suppose that $d_i >0$. 
Then $A_{i,d}$ is regular at $q=0$ and 
has a pole of order at most $d_i-1$ at $q=\infty$.  
Therefore it has to be a polynomial of degree $d_i-1$ in $q$. 
Suppose that $d_i=0$. We have 
\[
A_{i,d} = \sum_{d= d' + d''} S_{d'} P_i^{-1} T_{d''}
\]
The right-hand side is regular at $q=0$ and $\infty$ and the left-hand side 
is regular away from $0,\infty$. 
Hence both sides are independent of $q$ and 
$A_{i,d} = \sum_{d = d' +d''} S_{d'} 
P_i^{-1} T_{d''}|_{q=\infty} = P_i^{-1}\delta_{d,0}$. 
\end{proof} 

\begin{remark} 
\label{rem:logQ} 
We can also eliminate the factor $P_i^{-1}$ in \eqref{eq:A_S} by 
modifying the fundamental solutions. Setting 
$\tS = S (\prod_{i=1}^rP_i^{\log Q_i/\log q})$ 
and $\tT = (\prod_{i=1}^r P^{-\log Q/\log q}) T$, 
we have  
\[
\cA_i \circ \tS = \tS \circ q^{Q_i\partial_{Q_i}}, \qquad 
q^{Q_i\partial_{Q_i}} \circ \tT = \tT \circ \cA_i.  
\]
Because $P_i-1$ is nilpotent, we can define $P_i^{\log Q_i/\log q}$ 
as the binomial expansion of $(1 + (P_i-1))^{\log Q_i/\log q}$. 
\end{remark}

\begin{proposition} 
\label{prop:J_generator}
The $J$-function determines the difference-differential system 
in the following sense: setting 
$\tJ = (\prod_{i=1}^r P_i^{-\log Q_i/\log q}) J$, we have 
\[
f(q, Q, t, q^{Q_i\partial_{Q_i}}, (1-q) \partial_\alpha) \tJ =0 \ 
\Longleftrightarrow \ 
f(q, Q, t, \cA_i, \nabla^q_\alpha) \Phi_0 = 0 
\]
for any difference-differential operator $f 
\in \Q[q,q^{-1}][\![Q,t]\!]\langle 
q^{Q_i \partial_{Q_i}}, (1-q) \partial_\alpha\rangle$.  
Setting $q=1$, this equation gives rise to the relation 
$f(1, Q,t, A_{i,\rm com}, \Phi_\alpha\bullet) \Phi_0 =0$ 
in the quantum $K$-ring. 
\end{proposition} 
\begin{proof} 
Apply the operator $f$ to the formula $\tJ =(1-q) \tT \Phi_0$ 
and use Theorem \ref{thm:fund_sol} and \eqref{eq:A_S} 
(see also Remark \ref{rem:logQ}). 
\end{proof} 

\subsection{Shift operator from graph space invariants}
\label{subsec:graph}
In this section we give a geometric construction of 
the $q$-shift operator $\cA_i$ assuming that $P_i$ 
is represented by a very ample line bundle. The construction 
is close in spirit to the work 
of Givental-Lee \cite{Givental-Lee}. 
The variable $q$ is identified with a generator of 
$K_{S^1}({\rm pt})$, or the character of the 
natural representation of $\C^\times$. 

Let $GX_{n,d}$ denote the moduli space of genus-zero, 
degree $(d,1)$ stable maps to $X \times \C\Proj^1$ with $n$ 
marked points. 
This is the so-called \emph{graph space}. 
Let $\zeta$ a rational co-ordinate on $\C\Proj^1$ and consider 
the $\C^\times$-action on $\C\Proj^1$ 
given on points by $\zeta \mapsto q^{-1}\zeta$, $q\in \C^\times$. 
This induces a $\C^\times$-action on $GX_{n,d}$. 
Suppose that $P_i$ is the class of a very ample 
line bundle which is the pull-back of $\cO(1)$ 
by an embedding  $X \to \C\Proj^N$. 
Then we have a $\C^\times$-equivariant morphism  
\begin{equation} 
\label{eq:mu}
\mu_i \colon GX_{n,d} \to GX_{0,d} \to G\C\Proj^N_{0,d_i} 
\to Q\C\Proj^N_{d_i}  
\end{equation} 
where $d_i = \pair{p_i}{d}$ and 
$Q\C\Proj^N_{d_i}\cong \C\Proj^{(N+1)(d_i+1)-1}$ 
is the quasi map space from $\C\Proj^1$ to $\C\Proj^N$ of degree $d_i$: 
\[
Q\C\Proj^N_{d_i} = \left\{ (f_1(\zeta),\dots, f_{N+1}(\zeta)) 
\in \C[\zeta]^{\oplus(N+1)}
\setminus \{0\}: \deg f_j \le d_i, \ \forall j \right\}\big/\C^\times. 
\]
The induced $\C^\times$-action on $Q\C\Proj^N_{d_i}$ is 
given by $f_j(\zeta) \mapsto f_j(q \zeta)$. 
The first map in \eqref{eq:mu} is the forgetful morphism of marked points, 
the second map is induced by the embedding $X\to \C\Proj^N$ 
and the third map is the map of Givental 
\cite[Main Lemma]{Givental:equivariantGW} described as follows. 
Let $f\colon C \to \C\Proj^1 \times \C\Proj^N$ 
be a stable map in $G\C\Proj^N_{d_i}$ and 
let $C = C_0 \cup \bigcup_{a=1}^n C_a$ be 
the irreducible decomposition of $C$ such 
that $f(C_0)$ is of degree $(k_0, 1)$ and 
$f(C_a)$ is of degree $(k_a,0)$ for $1\le a \le n$. 
Then $f(C_0)$ is the graph of a function $\C\Proj^1\ni \zeta 
\mapsto [h_1(\zeta),\dots,h_{N+1}(\zeta)] \in \C\Proj^N$ 
for some polynomials $h_1,\dots, h_{N+1}$ 
of degree $\le k_0$ (without common factors) 
and $f(C_a)$ is contained in $\C\Proj^N \times \{\zeta_a\}$ 
for some $\zeta_a\in \C\Proj^1$. 
We assign to $f$ a point in $Q\C\Proj^N_{d_i}$ represented by 
the $(N+1)$-tuple of polynomials 
$(h_j(\zeta) \prod_{a: \zeta_a \neq \infty} 
(\zeta-\zeta_a)^{k_a})_{j=1}^{N+1}$.  

Consider the line bundle corresponding to $\cO(1)$ on $Q\C\Proj^N_{d_i}$: 
\[
\{(f_1(\zeta),\dots,f_{N+1}(\zeta), v) \in 
\left(\C[\zeta]^{\oplus (N+1)} \setminus \{0\}\right)\times \C : 
\deg f_j \le d_i, \ \forall j \}
\big/\C^\times 
\]
and equip it with the $\C^\times$-linearization: 
\[
[f_1(\zeta),\dots,f_{N+1}(\zeta), v]\mapsto [f_1(q \zeta), 
\dots, f_{N+1}(q \zeta), v], \qquad q\in \C^\times. 
\]
Let $\hP$ denote the corresponding 
$\C^\times$-equivariant line bundle. We set 
\[
\hP_i := \mu_i^*(\hP). 
\]
The $\C^\times$-fixed loci of the graph space 
are given by: 
\[
(GX_{n,d})^{\C^\times} 
= \bigsqcup_{
\substack{ \{1,\dots,n\} = I \sqcup J \\ 
d= d' + d'' }}
X_{0, I \cup \bullet, d'} \times_X X_{0, J \cup \circ, d''} 
\]
where the fibre product is taken with respect to the 
evaluation maps $\ev_\bullet$, $\ev_\circ$ associated 
to the additional markings $\bullet$ and $\circ$. 
An element of $X_{0,I \cup \bullet, d'} 
\times_X X_{0, J\cup \circ, d''}$ 
is represented by a stable map consisting of ``trivial" horizontal component 
(mapping to $\{x\}\times \C\Proj^1$ for some $x\in X$), 
a degree $d'$ stable map to $X\times \{0\}$ 
and a degree $d''$ stable map to $X\times \{\infty\}$. 
The normal bundle of the fixed component 
$X_{0,I \cup \bullet, d'} \times_X X_{0, J\cup \circ, d''}$ is 
\[
L_\bullet^{-1}\otimes T_0\C\Proj^1 
\oplus \cO\otimes T_0\C\Proj^1 \oplus 
\cO \otimes T_\infty \C\Proj^1 \oplus 
L_\circ^{-1} \otimes T_\infty \C\Proj^1  
\]
with $\C^\times$-characters $(q^{-1},q^{-1},q,q)$. 
The line bundle $\hP_i$ restricted to 
$X_{S_1 \cup \bullet, d_1} \times_X X_{X_2\cup \circ, d_2}$ 
is given by 
\[
\hP_i \big|_{X_{0, I \cup \bullet, d'} \times_X X_{0, J \cup \circ, d''}} 
\cong  q^{-d'_i} \ev_\bullet^*(P_i)  
\]
with $d_i' = \pair{p_i}{d'}$, where $q^{d'_i}\in K_{\C^\times}({\rm pt})$ 
denotes the rank-one $\C^\times$-representation 
of weight $d_i'$ (whose character is $q^{d_i'}$). 

Let $[0]$, $[\infty]$ denote the classes in $K_{\C^\times}(\C\Proj^1)$ 
represented by the structure sheaf on 
the fixed points $0$ and $\infty$ respectively. 
Then we have 
\[
[0]|_0 = 1 - q, \quad [\infty]|_{\infty} = 1-q^{-1}
\]
and $[0]|_{\infty} = [\infty]|_0 = 0$. 

Using the above facts and the localization formula, 
we obtain the following result. 
A similar computation appeared in \cite{Givental-Lee}. 
\begin{proposition}
\label{prop:graphinvariants} 
Let $\ev_i\colon GX_{n+2,d} \to X\times \C\Proj^1$ be 
the evaluation map.  
We have 
\[
G(A_i \Phi_\alpha, \Phi_\beta) 
= \sum_{n,d}
\frac{Q^d}{n!} 
\chi_{\C^\times}\left( GX_{n+2,d}, 
\cO^{\rm vir} \otimes \hP_i^{-1} \ev_0^* ([0]\Phi_\alpha)
\prod_{i=2}^{n+1}\ev_i^*(t)  
\ev_{n+2}^*([\infty]\Phi_\beta) \right)  
\]
where $t\in K(X)$. 
\end{proposition} 
\begin{proof} 
We calculate the right-hand side by the localization formula. 
Because of the insertions $[0]$ and $[\infty]$, 
we only need to consider fixed components where 
the 1st marking maps to $X\times \{0\}$ and 
the last marking maps to $X\times \{\infty\}$. 
The right-hand side equals: 
\[
\sum_{\substack{ \{2,\dots,n+1\} = I \sqcup J \\ 
d = d' + d'' }}
\frac{Q^d}{n!} 
\chi 
\left(\cO^{\rm vir} \otimes 
\frac{q^{d'_i}\ev_\bullet^*(P_i^{-1})  
 (1-q)\ev_1^*(\Phi_\alpha) 
 \prod_{j=2}^{n+1}\ev_j^*(t) 
(1-q^{-1}) \ev_{n+2}^*(\Phi_\beta)}
{(1-q)(1-q^{-1}) (1-q L_\bullet) (1  - q^{-1} L_\circ)}  
\right) 
\]
where the $\cO^{\rm vir}$ is the virtual structure 
sheaf on the fixed component 
$X_{0, 1\cup I \cup \bullet,d'}\times_X 
X_{0, \circ \cup J \cup n+2, d''}$. 
By the functoriality of virtual structure sheaves 
\cite[Proposition 7]{Lee:foundation}, 
we can rewrite this as 
\[
\sum_{n', n'' ,d', d''} q^{d'_i} \frac{Q^{d'+d''}}{n'!n''!}
\corr{\frac{P_i^{-1} \Phi_\epsilon}{1- q L}, 
\Phi_\alpha, t,\dots,t}_{0,n'+2,d'}^X 
g^{\epsilon\rho} 
\corr{\frac{\Phi_\rho}{1-q^{-1} L}, 
\Phi_\beta, t,\dots,t}_{0,n''+2,d''}^X 
\]
where $g^{\epsilon\rho}$ are the coefficients of the 
inverse matrix of $(g_{\epsilon\rho})$. 
This equals 
\[
g\left(
P_i^{-1}q^{Q_i\partial_{Q_i}} T \Phi_\alpha, 
\ov{T} \Phi_\beta \right) 
= g\left(  T A_i \Phi_\alpha, \ov{T} \Phi_\beta\right)
= G(A_i \Phi_\alpha, \Phi_\beta) 
\]
by \eqref{eq:A_S} and Proposition \ref{prop:unitarity}. The conclusion follows. 
\end{proof} 

\begin{remark}
Proposition \ref{prop:graphinvariants} gives an alternative proof 
of the fact that $A_i$ lies in $\End(K(X))\otimes \Q[q,q^{-1}][\![Q,t]\!]$ 
when $P_i$ is very ample. 
\end{remark} 

\section{Reconstruction}
\label{sec:rec}

By Corollary \ref{cor:Acom}, $K(X) \otimes \Q[\![Q]\!]$ 
has the structure of a $\Q[\![Q]\!][a_1,\dots,a_r]$-module 
where we let $a_i$ act by $A_{i,\rm com}|_{t=0}$.  
Let $\Q(\!(Q)\!)$ denote the fraction ring of $\Q[\![Q]\!] 
= \Q[\![Q_1,\dots,Q_r]\!]$. 
The main result in this section is:  
\begin{theorem} 
\label{thm:reconstruction} 
Suppose that $K(X)\otimes \Q(\!(Q)\!)$ is a cyclic 
$\Q(\!(Q)\!)[a_1,\dots,a_r]$-module (where $a_i$ 
acts by $A_{i,\rm com}|_{t=0}$). 
Then the small $q$-shift operators $\cA_i|_{t=0}$, $1\le i\le r$ 
reconstruct the quantum $K$-theory at genus zero, i.e.~they 
reconstruct the potential $\cF$, the metric $G$, the quantum 
products $\Omega_\alpha = (\Phi_\alpha \bullet)$, 
$0\le \alpha \le N$, the big $q$-shift operators 
$\cA_i = A_i q^{Q_i \partial_{Q_i}}$, $1\le i\le r$ 
and the fundamental solutions $S$, $T$.
\end{theorem} 

\begin{remark} 
The fundamental solution $T=S^{-1}$ recovers the Givental's 
Lagrangian cone \cite{Givental-Tonita} as the union 
$\bigcup_t (1-q)T_t(K(X)\otimes \Q[q,q^{-1}][\![Q]\!])$, and therefore 
we reconstruct all genus-zero descendant $K$-theoretic 
GW invariants in the above theorem. 
\end{remark} 

We remark that the small $q$-shift 
operators $\cA_i|_{t=0}$, $i=1,\dots,r$ and 
the small $J$-function $J|_{t=0}$ 
have the same amount of information when the classical $K$-ring 
$K(X)$ is generated by line bundles as a ring.

\begin{lemma} 
\label{lem:initial_condition} 
Suppose that $K(X)$ is generated by line bundles as a ring.  
Then one of the following data determines the other two: 
\begin{itemize} 
\item[(1)] the small $J$-function $J|_{t=0}$;  
\item[(2)] the fundamental solutions $S|_{t=0}$, $T|_{t=0}$; 
\item[(3)] the small $q$-shift operators $\cA_i|_{t=0}$.  
\end{itemize} 
\end{lemma} 

If the classical $K$-ring $K(X)$ is generated by line bundles as a ring, 
the assumption in Theorem \ref{thm:reconstruction} 
is automatically satisfied since we have $A_{i,\rm com} = P_i^{-1} 
+ O(Q)$ (see Proposition \ref{prop:A_expansion}).  
Therefore Theorem \ref{thm:reconstruction} have the following corollary: 
\begin{corollary}[Lee-Pandharipande \cite{Lee-Pandharipande}] 
\label{cor:Lee-Pandharipande} 
Suppose that $K(X)$ is generated by line bundles as a ring. 
Then the quantum $K$-theory at genus zero can be reconstructed 
from the small $J$-function $J|_{t=0}$. 
\end{corollary}  

The small $J$-function is known for complete flag manifolds 
\cite{Givental-Lee} and certain complete intersections in 
products of projective spaces \cite{Givental-Tonita}.  
We consider examples of $\C\Proj^1$, $\C\Proj^2$ 
and $\Fl_3$ in the next section. 

\begin{remark} 
The condition that $K(X)\otimes \Q(\!(Q)\!)$ is a cyclic 
$\Q(\!(Q)\!)[a_1,\dots,a_r]$-module is stronger than 
the condition that $K(X)$ is generated by line bundles. 
In this paper we do not give an example where 
$K(X)$ is not generated by line bundles but $K(X) \otimes \Q(\!(Q)\!)$ 
is a cyclic $\Q(\!(Q)\!)[a_1,\dots,a_r]$-module. 
This probably happens for Grassmannians $\Gr(r,n)$ 
with $(n,r!) =1$ because in that case the quantum cohomology 
of $\Gr(n,r)$ is generated by divisors. 
\end{remark} 

\begin{proof}[{{\bf Proof of Lemma \ref{lem:initial_condition}}}] 
The data (2) determines (1) and (3): this is obvious from the 
definitions of $J$ and $\cA_i$. 
The data (3) determines (2): we can solve the difference equation 
$P_i^{-1}q^{Q_i\partial_{Q_i}} T = TA_i$ for $T$ 
(see \eqref{eq:A_S}) with the initial condition $T|_{t=0,Q=0} = \id$. 
See the proof of Proposition \ref{prop:convergence_S} for more details. 

It suffices to show that the data (1) determines (2). 
By the assumption, we have polynomials 
$F_0,\dots,F_N\in \Q[x_1,\dots,x_r]$ 
such that $F_\alpha(P_1^{-1},\dots,P_r^{-1}) = \Phi_\alpha$. 
By Theorem \ref{thm:fund_sol} and the 
definition of the $J$-function, we have: 
\[
\frac{1}{1-q} 
F_\alpha(P_1^{-1} q^{Q_1 \partial_{Q_1}},
\dots,P_r^{-1} q^{Q_r\partial_{Q_r}}) J\Big|_{t=0} 
=  T F_\alpha(\cA_1,\dots,\cA_r) \Phi_0 \Big|_{t=0}.
\]
We solve for $T|_{t=0}$ from this 
equation. 
Let $M$ denote the $(N+1)\times (N+1)$-matrix whose 
$\alpha$-th column is given by the left-hand side 
of this equation. This is determined by the small $J$-function. 
Let $U$ denote the $(N+1)\times (N+1)$-matrix 
with the $\alpha$-th column 
given by $F_\alpha(\cA_1,\dots,\cA_r) \Phi_0|_{t=0}$. 
The above equation can be written as 
\begin{equation} 
\label{eq:Birkhoff}
M = T|_{t=0} U. 
\end{equation} 
We view this as an analogue of Birkhoff factorization 
(see, e.g.~\cite{Coates-Givental, Guest}) appearing in 
quantum $D$-module theory. 
Expand $M, T|_{t=0}, U$ in $Q$-series as 
$M = \sum_d M_d Q^d$, $T|_{t=0}= \sum_{d} T_d Q^d$ 
and $U = \sum_{d} U_d Q^d$. 
As functions in $q$, $T_d$ has poles only at 
roots of unity and vanishes at $q=\infty$ if $d\neq 0$. 
Also $U_d$ is a Laurent polynomial in $q$. 
Using $T_0 =U_0= \id$, we have the equation 
\[
M_d = T_d + U_d+ \sum_{d = d' + d'', d'\neq 0, d'' \neq 0} T_{d'} U_{d''}. 
\]
This equation can be solved recursively. 
In fact, choose an ample class $\omega$ and suppose that 
we know $T_{d'}$, $U_{d'}$ for all ${d'}$ with $\pair{\omega}{d'}
< \pair{\omega}{d}$. Then the above equation 
determines $T_d + U_d$, and the above mentioned properties 
of $T_d$, $U_d$ as functions in $q$ determine each 
summand $T_d$, $U_d$ uniquely. 
Therefore $T|_{t=0}$ is determined by the small 
$J$-function. This completes the proof. 
\end{proof} 

Let $\Omega_\alpha(Q,t)\in \End(K(X))\otimes \Q[\![Q,t]\!]$ 
be the matrix of quantum multiplication 
by $\Phi_\alpha$. 
The differential equations for the $q$-shift 
operators $\cA_i=A_i(q,Q,t) \, q^{Q_i\partial_{Q_i}}$ 
(see Remark \ref{rem:Lax}) 
can be written as: 
\begin{equation}
\label{system}
(1-q)\partial_\alpha A_i(q,Q,t) = A_i(q,Q,t) \Omega_\alpha(q^{e_i}Q,t) -
\Omega_\alpha(Q,t) A_i(q,Q,t), \quad 0\leq \alpha\leq N,
\end{equation} 
where  $\partial_\alpha = \partial/\partial t^\alpha$ and 
\[
q^{e_i}Q := (Q_1,\dots,Q_{i-1},qQ_i,Q_{i+1},\dots, Q_r).
\]
The following lemma is a key to prove Theorem \ref{thm:reconstruction}. 
\begin{lemma} 
\label{lem:rational_expression}
Suppose that $K(X)\otimes \Q(\!(Q)\!)$ is a cyclic 
$\Q(\!(Q)\!)[a_1,\dots,a_r]$-module. 
Then entries of the matrix $\Omega_\alpha$ can be expressed 
as rational functions of the entries of the matrices $A_{i,\rm com}
=A_i|_{q=1}$, $i=1,\dots,r$. 
The rational functions here have coefficients in $\Q$ and 
are determined only by $a_i=A_{i,\rm com}|_{t=0}$, 
$1\le i\le r$.  
\end{lemma} 
\begin{proof} 
By the assumption, there exists an element $v\in K(X) \otimes \Q(\!(Q)\!)$ 
such that $\Q(\!(Q)\!)[a_1,\dots,a_r] v = K(X)\otimes \Q(\!(Q)\!)$.  
One can find polynomials $F_\alpha(x_1,\dots,x_r) 
\in \Q[x_1,\dots, x_r]$ with rational coefficients, 
$\alpha=0,\dots, N$ 
such that $F_\alpha(a_1,\dots,a_r) v$, $\alpha=0,\dots, N$ 
form a $\Q(\!(Q)\!)$-basis of $K(X) \otimes \Q(\!(Q)\!)$. 
Set $B_\alpha = F_\alpha(a_1,\dots,a_r)\Phi_0$. 
Then one has $[v\bullet B_\alpha]_{t=0}= 
F_\alpha(a_1,\dots,a_r) [v\bullet \Phi_0]_{t=0} = F_\alpha(a_1,\dots,a_r) v$ 
by Corollary \ref{cor:Acom}. 
Therefore $\{B_\alpha: 0\le \alpha\le N\}$ is linearly 
independent over $\Q(\!(Q)\!)$ and form a basis of $K(X)\otimes \Q(\!(Q)\!)$. 
Since $A_{i,\rm com}$ commutes with 
$\Omega_\alpha$ (Corollary \ref{cor:Acom}), we have 
\begin{align*} 
\Omega_\alpha 
F_\beta(A_{1,\rm com},\dots,A_{r,\rm com}) \Phi_0
& = F_\beta(A_{1,\rm com},\dots, A_{r,\rm com}) \Omega_\alpha \Phi_0 \\ 
& = F_\beta(A_{1,\rm com},\dots, A_{r, \rm com}) \Phi_\alpha 
\end{align*} 
where we used $\Omega_\alpha \Phi_0 = \Phi_\alpha \bullet \Phi_0 = \Phi_\alpha$. 
This should be viewed as a matrix equation: 
\begin{equation} 
\label{eq:Omega_determined} 
\Omega_\alpha 
\begin{pmatrix} 
\vert &  & \vert \\ 
F_0(A_{*, \rm com}) \Phi_0 & \cdots & F_N(A_{*,\rm com}) \Phi_0 \\
\vert &  & \vert 
\end{pmatrix} 
= 
\begin{pmatrix} 
\vert &  & \vert \\ 
F_0(A_{*, \rm com}) \Phi_\alpha & \cdots & F_N(A_{*,\rm com}) \Phi_\alpha \\
\vert &  & \vert 
\end{pmatrix}. 
\end{equation} 
The matrix 
$[F_0(A_{*, \rm com}) \Phi_0 , \dots, F_N(A_{*,\rm com}) \Phi_0]$ 
in the left-hand side is invertible since it coincides with 
$[B_0,\dots,B_N]$ at $t=0$.  
By inverting this matrix, we obtain rational function 
expressions for entries of $\Omega_\alpha$. 
\end{proof} 

\begin{proof}[{{\bf Proof of Theorem \ref{thm:reconstruction}}}] 
By Lemma \ref{lem:rational_expression}, $\Omega_\alpha$ is a 
rational function in all the entries of $A_{i,\rm com}=A_i|_{q=1}$, $i=1,\dots,r$. 
Therefore the equation \eqref{system} gives a closed 
differential equation for unknown functions $A_i$, $i=1,\dots,r$.  
Expanding  $A_i$ in power series of $t$, we write 
$A_i = \sum_{n\ge 0} A_i^{(n)}$, 
$A_i^{(\le n)} = \sum_{k=0}^n A_i^{(k)}$, where $A_i^{(n)}$ 
is the degree $n$ homogeneous part with respect to $t$. 
Then equation \eqref{system} determines 
$A_i^{(n+1)}$ uniquely in terms of $A_i^{(\le n)}$, since the degree 
$\le n$ part of $\Omega_\alpha$ is a rational function of $A_i^{(\le n)}$ 
by Lemma \ref{lem:rational_expression}. 
Therefore $A_i$, $i=1,\dots,r$ and $\Omega_\alpha$, $\alpha=0,\dots,N$ 
are uniquely determined by $A_i|_{t=0}$. 

The fundamental solution $T$ is uniquely determined from $A_i$ 
as a solution to the difference equation 
$P_i^{-1} q^{Q_i \partial_{Q_i}} T = T A_i$ 
with the initial condition $T|_{t=0, Q=0} = \id$ 
(see the proof of Proposition \ref{prop:convergence_S}). 
The metric $G$ is determined from $T$ as 
$G(\Phi_\alpha,\Phi_\beta) = g(\ov{T}\Phi_\alpha, T \Phi_\beta)$ 
(Proposition \ref{prop:unitarity}). By the string equation 
one has $F = F_{000} - (\chi(\cO) + \chi(t) + \frac{1}{2}\chi(t\otimes t))$. 
Since $F_{000} = G_{00}$, $F$ is determined by the metric $G$. 
\end{proof} 

\begin{remark} 
\label{rem:algorithm} 
Suppose that $K(X)$ is generated by line bundles as a ring. 
In this case we can give a more effective reconstruction algorithm: 
we can determine the coefficients of $A_i$, $\Omega_\alpha$ 
in front of $Q^d$ (for a fixed $d$) as polynomials in $t^\alpha$ 
and $e^{t^\alpha}$. 
This method will be also illustrated in \S \ref{sec:example} below.  
Expand $A_i$, $\Omega_\alpha$ as 
\[
A_i = \sum_{d\in \Eff(X)} \sum_{k=0}^{\max(d_i-1,0)} 
A_{i,d,k}(t)  (1-q)^k Q^d, \quad 
\Omega_\alpha(Q,t) = \sum_{d\in \Eff (X) } 
\Omega_{\alpha,d}(t) Q^d. 
\]
Let $\omega$ be an ample class, so that $\pair{\omega}{d}\geq 0$ for
all $d\in \Eff(X)$ with equality if and only if $d=0$.  
Firstly we refine Lemma \ref{lem:rational_expression} as: 
\emph{the entries of the matrix $\Omega_{\alpha,d}$ with 
$0\leq \alpha \leq N$, $d\in \Eff(X)$ are polynomial 
expressions in the entries of the matrices
$A_{i,d',0}$ for which} 
\begin{equation*}
1\leq i\leq r  \quad\text{and} \quad 
\left\{ 0<\pair{\omega}{d'} < \pair{\omega}{d}  
\ \  \text{or} \ \ d'=d \right\}.  
\end{equation*}
\emph{Moreover, if each entry of $A_{i,d',0}$ is assigned degree 
$d'\in \Eff(X)$, then the above polynomials are homogeneous of degree $d$.} 
This can be proved by taking $F_\alpha(x_1,\dots,x_r)$ in the proof of 
Lemma \ref{lem:rational_expression} to be polynomials 
such that $P_\alpha(P_1^{-1},\dots,P_r^{-1}) = \Phi_\alpha$. 
Suppose by induction that we know $A_{i,d',k}(t)$ for all $i=1,2,\dots,r$, 
$k\geq 0$ and $d'\in \Eff(X)$ with $\pair{\omega}{d'} <\kappa$.  
We want to determine $A_{i,d,k}$. 
Comparing the coefficients in front of $Q^d$ and $(1-q)^k$ in
\eqref{system} we get (assuming that $k\geq 1$)
\begin{equation}
\label{ode}
\partial_\alpha A_{i,d,k-1} = \sum_{d'+d''=d} \,\left(
[A_{i,d',k},\Omega_{\alpha,d''}] +  \sum_{k'+k''=k, \, k''\geq 1} (-1)^{k''}{d_i''
  \choose k''}A_{i,d',k'}\, \Omega_{\alpha,d''}
\right),
\end{equation}
where we may assume that the summation indices $k'$ and $k''$ satisfy
the inequalities $
0\leq k'\leq \max(d_i'-1,0)$, $1\leq k''\leq d_i''$ 
because if otherwise, then $A_{i,d',k'}=0$. 
The matrix $\Omega_{\alpha,d}$ can be decomposed as 
$\Omega_{\alpha,d}^{(1)} + \Omega_{\alpha,d}^{(2)}$, 
where the entries of $\Omega_{\alpha,d}^{(1)}$ are linear combinations
of the entries of $A_{j,d,0}$ for $j=1,2,\dots,r$, while the entries
of the second matrix are polynomials of the entries of $A_{j,d',0}$
for $1\leq j\leq r$ and $\pair{\omega}{d'}<\kappa$. 
The contribution on the right-hand side of \eqref{ode} 
of the terms that we have not  determined yet
on the previous steps is the following: 
\begin{equation}
\label{eq:linear-op}
[A_{i,d,k},\Omega_{\alpha,0}] 
+(-1)^k{d_i\choose k} A_{i,0,0}\, \Omega^{(1)}_{\alpha,d}.
\end{equation}
Let $Y$ be the vector formed by all entries of all matrices $A_{i,d,k}$ for
$1\leq i\leq r$ and $0\leq k\leq d_i-1$.  
It lies in the vector space
$ \mathfrak{X}_d: = \prod_{i=1}^r\,\prod_{k=0}^{d_i-1}\, 
\End(K(X))\otimes \C$. 
Formula \eqref{eq:linear-op} defines a linear operator
$L_\alpha:\mathfrak{X}_d\to \mathfrak{X}_d$, so the differential
equation \eqref{ode} assumes the form
\[
\partial_\alpha\, Y = L_\alpha(Y) + B_\alpha(t),\quad 0\leq \alpha\leq N,
\]
where $B_\alpha(t)$ is some explicit function determined from the
previous steps. This differential equation has a unique solution 
when the initial condition $A_i|_{t=0}$ is given. 
Solving the differential equation for $Y(t)$ 
is an algebraic problem, which depends on the spectral 
decomposition of the linear operator $L_\alpha$ (see \cite{In}). 
In particular, our argument inductively shows that
the entries of $A_{i,d,k}(t)$ are polynomials in $t^0,\dots,t^N$ and
$e^{t^0},\dots,e^{t^N}$. This is compatible with 
the polynomiality in Corollary \ref{cor:polynomiality}. 
\end{remark}

\section{Examples} 
\label{sec:example} 

In this section we illustrate the reconstruction by computing 
the potential of $\C\Proj^1$ up to degree $4$ 
and of $\C\Proj^2$ up to degree $3$, which will give us
the opportunity to go through all the steps of the algorithm 
outlined in Remark \ref{rem:algorithm}. 
We also compute the small quantum product of 
the complete flag manifold $\Fl_3$ and show its 
finiteness. 

\subsection{$q$-shift operator for projective spaces}  
The $K$-ring of $\C\Proj^{N}$  is
\begin{align*}
K(\C\Proj^N)= \frac{\Q[P,P^{-1}]}{
\left( (1-P^{-1})^{N+1}\right) },
\end{align*}
where $P=\mathcal{O}(1)$. 
The small $J$-function of $\C\Proj^N$ is known to be 
\cite{Lee:thesis, Givental-Lee, Givental-Tonita} 
\begin{align}
\label{eq:J_projective}
 J = (1-q)\sum_{d=0}^\infty 
\frac{Q^d}{\prod^d_{r=1}(1-P^{-1}q^r)^{N+1}}.
\end{align}
It satisfies the equation
\begin{align}
\label{fdj}
\left(1-P^{-1}q^{Q\partial_Q}\right)^{N+1} J = QJ. 
\end{align} 
\begin{proposition} 
\label{prop:A_projective} 
Consider the $q$-shift operator $\cA =A  q^{Q\partial_Q}$ 
for the projective space $\C\Proj^N$. 
The endomorphism $A|_{t=0}$ is represented by the matrix 
\[
I - 
\begin{pmatrix} 
0 & 0 &   \cdots & 0 & Q \\ 
1 & 0 &  \cdots & 0  & 0 \\
0 & 1 &  \cdots & 0 & 0 \\
\vdots &   \vdots & \ddots & \vdots & \vdots\\ 
0 & 0 &  \cdots & 1 & 0
\end{pmatrix} 
\]
in the basis $(1-P^{-1})^\alpha$, $\alpha=0,\dots,N$, 
where $I$ is the identity matrix. 
\end{proposition} 
\begin{proof} 
We follow the construction of the difference operator $A|_{t=0}$ 
from the small $J$-function in Lemma \ref{lem:initial_condition}.  
Applying the difference operator $\left(1-P^{-1}q^{Q \partial_Q}\right)^\alpha$, 
$\alpha=0,\dots,N$ to $(1-q)^{-1} J$, we obtain: 
\begin{equation} 
\label{eq:T_alpha}
(1-P^{-1}q^{Q \partial_Q})^\alpha (1-q)^{-1} J  
= (1-P^{-1})^{\alpha} + 
\sum_{d=1}^\infty 
\frac{Q^d}{(1-q^d P^{-1})^{N+1-\alpha} 
\prod_{k=1}^{d-1} (1-q^k P^{-1})^{N+1}}. 
\end{equation} 
The coefficient of $Q^d$ of this expression 
is a rational function in $q$ 
which has poles only at roots of unity. 
Also this equals $(1-P^{-1})^\alpha$ for $q=\infty$. 
Therefore the Birkhoff-type factorization discussed 
in \eqref{eq:Birkhoff} is trivial in this case (i.e.~$U$ is the identity matrix), 
and the $\alpha$-th column of the fundamental solution $T|_{t=0}$ is 
given by $(1-P^{-1}q^{Q \partial_Q})^\alpha J/(1-q)$. 
The $q$-difference equation \eqref{fdj} gives the 
above presentation for $A|_{t=0} = 
T^{-1} P^{-1}q^{Q\partial_Q} T|_{t=0}$. 
\end{proof} 

The commutativity of the quantum product and the $q$-shift 
operator yields the following corollary 
(see the proof of Lemma \ref{lem:rational_expression}): 
\begin{corollary}[\cite{Buch-Mihalcea}] 
\label{cor:CP_product} 
The small quantum $K$-product of $\C\Proj^N$ is 
given by 
\[
(1-P^{-1})^\alpha \bullet (1-P^{-1})^\beta \Big|_{t=0} 
= \begin{cases} 
(1-P^{-1})^{\alpha+\beta} & \text{if } \alpha+\beta\le N \\ 
 Q & \text{if } \alpha+\beta =N+1 
 \end{cases} 
\]
where $0\le \alpha,\beta \le N$.  In particular 
it is finite as a power series in $Q$. 
\end{corollary} 

Using the formula \eqref{eq:T_alpha} for $T (1-P^{-1})^\alpha|_{t=0}$ 
in the above proof and Proposition \ref{prop:unitarity}, 
we obtain  the pairing $G$ of $\C\Proj^N$ at $t=0$ as follows: 
\begin{corollary} 
The quantum $K$-theory pairing of $\C\Proj^N$ 
is given by $G(\Phi_\alpha,\Phi_\beta) |_{t=0} = 
g_{\alpha\beta} + Q/(1-Q)$, where $\Phi_\alpha = (1-P^{-1})^\alpha$ 
and $g_{\alpha\beta} = 
1$ if $\alpha+\beta \le N$ and $g_{\alpha\beta}=0$ otherwise.  
\end{corollary} 

\subsection{$\C\Proj^1$ invariants} 
Consider $X=\C\Proj^1$, and denote by $\Omega(Q,t)$ 
the matrix of quantum multiplication by $P^{-1}$. 
We use $\{\Phi_0,\Phi_1\} = \{1, P^{-1}\}$ 
as a basis of $K(\C\Proj^1)$. 
The dependence of $\Omega$ and $A$ on $t^0$ is trivial 
so we will consider them as functions of one variable $t:=t^1$ 
--- the coordinate dual to $P^{-1}$ in the $K$-ring. 

Following the notation from Remark \ref{rem:algorithm}, 
we write 
\begin{align*}
A (q,Q,t) &=\sum_{d=0}^\infty 
A_{d}(q,t) Q^d 
= P^{-1} + \sum_{d=1}^\infty \sum^{d-1}_{i=0} A_{d,i}(t)(1-q)^i Q^d, \\
\Omega(Q,t) &= \sum_{d=0}^\infty \Omega_d(t)Q^d.
\end{align*} 
According to Proposition \ref{prop:A_projective}, 
the value of $A$ at $t=0$ is (with respect to the basis $\{1,P^{-1}\}$):  
\begin{align*}
A(q,Q,0)= 
\begin{pmatrix}
 0 & Q-1 \\
1 & 2 
\end{pmatrix}. 
\end{align*}
Since $\Omega(Q,t) 1 = P^{-1} \bullet 1 = P^{-1}$, 
the matrices in the above expansion have the form
\begin{align*}
A_{d,i}(t)=  
\begin{pmatrix}
      m_{d,i}(t) & n_{d,i}(t) \\
      p_{d,i}(t) & r_{d,i}(t) 
\end{pmatrix},
\quad
\Omega_d(t)= 
\begin{pmatrix}
     0 & y_d(t) \\
     0 & x_d(t) 
\end{pmatrix} 
\end{align*}
for $d>0$. 
In degree $d=0$ we know that 
\begin{align*}
   A_0 =\Omega_0 = (P^{-1}\otimes) = 
\begin{pmatrix}
    0 & -1 \\
    1 & 2 
\end{pmatrix}.
\end{align*}
Since $A_{\rm com} = A|_{q=1}$ commutes with 
$\Omega$ and $\Omega 1 = P^{-1}$, 
we have 
\[
\Omega 
\begin{pmatrix} 
\vert & \vert \\ 
1 & A_{\rm com} 1 \\ 
\vert & \vert 
\end{pmatrix}  
= 
\begin{pmatrix} 
\vert & \vert \\ 
P^{-1} & A_{\rm com} P^{-1} \\
\vert & \vert 
\end{pmatrix}.  
\]
Solving this equation for $\Omega$, we can express $y_d(t)$, $x_d(t)$ in terms of 
$m_{d',0}(t)$, $n_{d',0}(t)$, $p_{d',0}(t)$, $r_{d',0}(t)$ with 
$d'\le d$: for example 
\begin{align}
\label{eq:x1y1x2y2} 
\begin{split}  
\begin{pmatrix} 
y_1(t) \\ 
x_1(t) 
\end{pmatrix}
& =  
\begin{pmatrix} 
n_{1,0}(t) + p_{1,0}(t) \\ 
r_{1,0}(t) - m_{1,0}(t) - 2p_{1,0}(t) 
\end{pmatrix},   \\ 
\begin{pmatrix} 
y_2(t) \\ 
x_2(t) 
\end{pmatrix} 
& = 
\begin{pmatrix} 
n_{2,0}(t) + p_{2,0}(t) - p_{1,0}(t) (p_{1,0}(t)+  n_{1,0}(t))   \\ 
r_{2,0}(t) - m_{2,0}(t) - 2p_{2,0}(t) 
+ p_{1,0}(t) (2 p_{1,0} - r_{1,0}(t) + m_{1,0}(t)) 
\end{pmatrix}.
\end{split}  
\end{align} 
Now we use the Lax equation 
\begin{equation}
\label{eq:Lax} 
(1-q) \parfrac{A}{t} = A \left( q^{Q\partial_Q} \Omega\right)  - \Omega A.  
\end{equation} 
Expanding it in the powers of $Q$ and $(1-q)$, we obtain for 
degrees $d=1,2$: 
\begin{align*}
\parfrac{A_{1,0}}{t} & = - A_{0} \Omega_1 = 
\begin{pmatrix}
 0 &  x_1(t) \\ 
 0 & -y_1(t) - 2 x_1(t) 
\end{pmatrix}, \\ 
\parfrac{A_{2,0}}{t} & = - 2 A_0 \Omega_2 - A_{1,0} \Omega_1 
+ [A_{2,1}, \Omega_0],  \qquad 
\parfrac{A_{2,1}}{t} = A_0 \Omega_2. 
\end{align*} 
We replace the expressions \eqref{eq:x1y1x2y2} for $x_d(t)$, $y_d(t)$ 
in the right-hand side and solve for $m_{d,i}(t)$, $n_{d,i}(t)$, 
$p_{d,i}(t)$, $r_{d,i}(t)$. 
In degree $d=1$, the differential equation for $A_{1,0}$ immediately gives 
$m_{1,0}(t)=p_{1,0}(t)=0$, $n_{1,0}'(t) = r_{1,0}(t)$, 
$r_{1,0}'(t) = -n_{1,0}(t)-2r_{1,0}(t)$ and the solutions
\begin{align*}
 A_1(q,t) = \Omega_1(t) = 
\begin{pmatrix}
    0 & (1+t)e^{-t} \\
    0 & -t e^{-t} 
\end{pmatrix}. 
\end{align*}
The computation in degree $d=2$ is also straight-forward, 
but lengthy and tedious. 
It is a good idea to incorporate a computer software 
(e.g. Mathematica or Maple). 
We leave the details to the reader and record the 
answer below
\begin{align*}
  A_{2,0}&=
\begin{pmatrix}
  \frac{-3+t^2 + e^{-2t}(3+6t+5t^2+2t^3)}{16}    
& \frac{2-2t+t^2 + e^{-2t}(-2-2t+7t^2+10t^3)}{16}\\
      \frac{2t-t^2 - e^{-2t}(2t+3t^2+2t^3)}{16}   
&\frac{-5+4t-t^2 + e^{-2t}(5+6t+3t^2-10t^3)}{16}
\end{pmatrix},
\\
A_{2,1}&= 
\begin{pmatrix}
              0    & \frac{t + e^{-2t}(-t-2t^2-2t^3)}{8}\\
             0   &\frac{1-t + e^{-2t}(1+t-2t^3)}{8}
\end{pmatrix},\\
\Omega_2 & = 
\begin{pmatrix}
   0    & \frac{1}{8}e^{-2t}(-1+e^{2t}-2t+2t^2+4t^3)\\
   0   &- \frac{1}{8}e^{-2t}(-1+e^{2t}-2t-2t^2+4t^3)
\end{pmatrix}.       
\end{align*}

The above computations suggest how our reconstruction 
method works in general. 
In fact for $\C\Proj^1$, taking advantage of the fact that 
our functions are in one variable, there is a simpler way to 
solve the system \eqref{eq:Lax}. 
Let us write the  coefficients in front of $Q^3(1-q)^i$ 
of the system \eqref{eq:Lax}, 
grouped by decreasing powers of $i=3,2,1,0$. 
  \begin{align*}
  \partial_t A_{3,2} & = - \Omega_0 \Omega_3 (t) \\
  \partial_t A_{3,1} & = [A_{3,2}(t),\Omega_0]-A_{2,1}(t)\Omega_1(t) + 
A_{1,0} (t)\Omega_2(t)+3A_0\Omega_3(t)  \\
  \partial_t A_{3,0} & = [A_{3,1}(t),\Omega_0]-A_{2,0}(t)\Omega_1(t) 
+A_{2,1}(t)\Omega_1(t) - 2A_{1,0}(t)\Omega_2(t) 
-\Omega_1(t)A_{2,1}(t)-3A_0\Omega_3(t) \\  0 
& = [A_{3,0}(t),\Omega_0]+ [A_{2,0}(t),\Omega_1(t)] 
+[A_{1,0}(t),\Omega_2(t)]+[A_0,\Omega_3(t)].
\end{align*}
This can be reduced into a single (matrix) equation of 
order $3$ for $\Omega_3(t)$: namely differentiate
the last equation three times and for $\partial_t^3 A_{3,0}(t)$ 
plug in the expression obtained by differentiating 
the equation in $A_{3,0}$ twice more. 
This gives us an equation involving 
$\partial_t^3\Omega_3, \partial_t^2\Omega_3$ and $\partial_t^2 A_{3,1}$. 
But $\partial_t^2 A_{3,1}$ can be replaced by differentiating 
the previous equation one more time and using the first equation 
we end up with an equation involving only 
$\partial_t^i \Omega_3$, $i\leq 3$ {\em with constant coefficients} 
and non-constant inhomogeneous terms. 
This can be solved for $y_3(t),x_3(t)$ (in fact it is an overdetermined system) 
and then going from top to bottom use each equation to solve for $A_{3,i}$. 

The same strategy works for any degree $d$. 
We write below the answers for $d=3,4$:
\footnotesize 
{\allowdisplaybreaks 
\begin{align*}
   y_3(t)= & 
\frac
{e^{-3t}(61-141t-684t^2-972t^3-162t^4+486t^5)
+e^{-t}(-189+135t+216t^2-54t^3)+128}{1728},\\
   x_3(t)= &
\frac{e^{-3t}(142+183t-252t^2-648t^3-648t^4+486t^5)
+e^{-t}(-270-27t+270t^2-54t^3)+128}{1728}, \\
   y_4(t)= &\frac{19683+e^{-t}(-54272+19456t+38400t^2-9216t^3)}{373248}+\\ 
   + &
\frac{e^{-2t}(34992-34992t-122472t^2+11664t^3+52488t^4-11664t^5)}
{373248}+\\
   + &
\frac{e^{-4t}(-403+29636t+112464t^2+97344t^3-67032t^4-194400t^5 
-77760t^6+62208t^7)}{373248}, \\
   x_4(t)= &\frac{21141+e^{-t}(-58368-14336t+47616t^2-9216t^3)}
{-373248}+\\ 
   + &\frac{e^{-2t}(29160+34992t-122472t^2-34992t^3+64152t^4-11664t^5)}
{-373248}+\\
   + & \frac{e^{-4t}(8067+11564t+83088t^2+139104t^3+
34056t^4-101088t^5-139968t^6+62208t^7)}{-373248}. 
   \end{align*}
}\normalsize 
To compute the $K$-theoretic GW invariants set 
\begin{align*}
g(t)= F_{000} (t P^{-1} ) = 1 -\frac{t^2}{2} 
+ \sum_{d,n} \corr{P^{-1},\dots,P^{-1}}_{0,n,d}^{\C\Proj^1} 
\frac{Q^d t^n}{n!}. 
\end{align*}
We claim that $g$ satisfies the equation
\begin{align} 
\label{eq:diffeq_g}
y(t) g(t) + x(t) g'(t)=g''(t) 
\end{align}
with $y(t) = \sum_{d=0}^\infty y_d(t) Q^d$ and 
$x(t) = \sum_{d=0}^\infty x_d(t) Q^d$. 
To begin with notice that
\begin{align*}
G\left(P^{-1}\bullet P^{-1}, 1\right) = F_{110}(tP^{-1}) 
= F_{11000}(tP^{-1}) = g''(t).
\end{align*} 
On the other hand, 
\[ 
   G \left (P^{-1}\bullet P^{-1}, 1\right) 
=  G\left( y(t)  + x(t) P^{-1},1\right)  
= y(t) g(t) + x(t) g'(t).
\]  
Equation \eqref{eq:diffeq_g} can be used 
to solve recursively in $d$ for the coefficients of $Q^d$ in $g$ 
(which we denote by $g_d$). 
We also know the initial condition $g(0) = 1/(1-Q)$ 
from the small $J$-function \eqref{eq:J_projective}. 
We get 
\footnotesize
{\allowdisplaybreaks 
\begin{align*} 
 g_1(t)= &1 \\
 g_2(t) = & \frac{e^{-t}(8+5t+t^2)+e^t(8-5t+t^2)}{16}\\
 g_3(t)= &\frac{2025-486t-162t^2+e^{-2t}(983+1428t+924t^2+324t^3+54t^4)
+e^t(2176-1152t+192t^2)}{5184} \\
 g_4(t)= &\frac{e^t (2651373-1463832t+274104t^2-11664t^3 )+
1687552-344064t-147456t^2}{5971968}+\\ 
 +&\frac{e^{-t}(1306368+676512t-93312t^3-23328t^4)}{5971968}+\\
 +&\frac{e^{-3t}(326675+766404t+864000t^2+613584t^3
+295488t^4+93312t^5+15552t^6)}{5971968}.
 \end{align*} 
}\normalsize 
Extracting coefficient of $t^n$ in $g_d(t),\, d=1,2,3,4,$ we get all the invariants
\begin{align*}
\la P^{-1},\ldots,P^{-1}\ra_{0,n,d},\quad d=1,2,3,4.
\end{align*}
For $d=1,2$ our answer agrees with a fixed-point localization
computation which gives the following numbers:
\begin{align*}
\la P^{-1},\ldots,P^{-1}\ra_{0,n,1}= & 0,  \\
\la P^{-1},\ldots,P^{-1}\ra_{0,n,2}= &\frac{1+(-1)^n}{2}
\binom{\frac{n}{2}-1}{2}.
\end{align*}  
\begin{remark} 
In Theorem \ref{thm:polynomiality} below, we show that 
the structure constants with fixed degree $d$ are polynomials 
in the variables $t^i$ and $e^{t^i}$  
under a suitable choice of basis. In the above computation 
we see the appearance of both $e^{-t}$ and $e^t$, but this 
is due to the current choice of the basis $\{1,P^{-1}\}$. If we take  
$\{1,1- P^{-1}\}$ as a basis and perform the corresponding 
change of variables, 
everything becomes a polynomial in $t^0,t^1, e^{t^0}, e^{t^1}$.  
\end{remark} 
  
\subsection{$\C\Proj^2$ invariants} 
We reconstruct the $K$-theoretic GW potential of $\C\Proj^2$ 
up to degree three. Since the method is completely parallel to 
the case of $\C\Proj^1$, we only give a brief sketch. 
We take $\{1, (1-P^{-1}), (1-P^{-1})^2\}$ as a basis 
of the $K$-ring of $\C\Proj^2$. 
Let $\{t^0,t,s\}$ denote the corresponding dual co-ordinates. 
As before, we set $t^0=0$ as the dependence on $t^0$ is determined 
by the string equation. 
We write $\cA = A q^{Q\partial_Q}$ for the $q$-shift operator 
and write $\Omega$ and $\Xi$ for the matrix of quantum multiplication 
by $1-P^{-1}$ and $(1-P^{-1})^2$ respectively. 
Because of the commutativity among $A_{\rm com}= A|_{q=1}$, 
$\Omega$ and $\Xi$, we can determine $\Omega$ and $\Xi$ 
as functions of $A_{\rm com}$ by the formula: 
\begin{align*} 
\Omega 
\begin{pmatrix} 
\vert &  \vert & \vert \\ 
1 &  A_{\rm com}1 & (A_{\rm com})^2 1 \\
\vert & \vert & \vert 
\end{pmatrix}
& = 
\begin{pmatrix} 
\vert & \vert & \vert \\ 
(1-P^{-1}) & A_{\rm com} (1-P^{-1}) & (A_{\rm com})^2(1-P^{-1})\\
\vert & \vert & \vert 
\end{pmatrix} 
\\
\Xi
\begin{pmatrix} 
\vert &  \vert & \vert \\ 
1 &  A_{\rm com}1 & (A_{\rm com})^2 1 \\
\vert & \vert & \vert 
\end{pmatrix}
& = 
\begin{pmatrix} 
\vert & \vert & \vert \\ 
(1-P^{-1})^2 & A_{\rm com} (1-P^{-1})^2 & (A_{\rm com})^2(1-P^{-1})^2 \\
\vert & \vert & \vert 
\end{pmatrix} 
\end{align*} 
Then the following Lax equations 
\[
(1-q) \parfrac{A}{t} = A (q^{Q\partial_Q}\Omega) - \Omega A, 
\quad 
(1-q) \parfrac{A}{s} = A (q^{Q\partial_Q}\Xi) - \Xi A 
\]
together with the initial condition (written in the basis 
$\{1,(1-P^{-1}), (1-P^{-1})^2\}$) 
\[
A|_{t=s=0} 
= \begin{pmatrix} 
1 & 0 & -Q \\ 
-1 & 1& 0 \\ 
0 & -1 & 1 
\end{pmatrix} 
\]
uniquely determine $A$. 
Up to degree one, the results for $A$, $\Omega$, $\Xi$ are given by
\begin{align*} 
A & = 
\begin{pmatrix} 
1 & -s e^t Q & -e^t Q \\ 
-1 & 1 + s(1+ t - \frac{1}{2}s)e^t Q & (t-s)e^t Q \\
0 & -1 + s(s-t)(1 + \frac{1}{2}t)e^t Q & 
1 + (s +ts - \frac{1}{2}t^2)e^tQ 
\end{pmatrix} 
+ O(Q^2)   \\
\Omega & = 
\begin{pmatrix} 
0 & s e^t Q & e^tQ \\ 
1 & s(-1-t + \frac{1}{2}s) e^t Q & (s-t) e^t Q \\ 
0 & 1 - s(s-t+\frac{1}{2}(s-t)t) e^t Q 
& -(s+ts-\frac{1}{2}t^2) e^t Q 
\end{pmatrix} 
+ O(Q^2)  \\ 
\Xi & = 
\begin{pmatrix} 
0 & e^tQ & 0 \\ 
0 & (s-t) e^t Q & e^tQ \\ 
1 & -(s+ts-\frac{1}{2}t^2)e^tQ & -t e^tQ 
\end{pmatrix} 
+ O (Q^2). 
\end{align*} 
One can continue the computation to higher degrees 
(by using a computer software), but the result is too long and 
we omit it. Instead we record  the $K$-theoretic 
GW potential for $\C\Proj^2$ up to degree three: 
\scriptsize
{\allowdisplaybreaks 
\begin{align*}
& F_{000}\left( t(1-P^{-1})+ s(1-P^{-1})^2 \right) 
= 1 + t + s + \frac{1}{2}t^2 + 
\frac{Q}{2}(2 + s^2 + 2s) e^t \\
& + 
Q^2 \left[ 
 \left( -\frac{1}{6}\,{s}^{3}t+{\frac {1}{192}}\,{t}^{4}s+
\frac{3}{16}\,{t}^{2}{s}^{2}+\frac{1}{12}\,{s}^{3}{t}^{2}- 
\frac{1}{24}\,{s}^{4}t - 
{\frac {21}{32}}\,ts- \frac{1}{24}\,{s}^{2}{t}^{3}-
{\frac {1}{3840}}\,{t}^{5}+
{\frac {15}{32}}\,{s}^{2}-
{\frac {5}{96}}\,{t}^{3}s+
{\frac {63}{512}}\,{t}^{2}-{\frac {93}{256}}\,t 
\right. \right. \\ 
& \left. \quad \qquad-{\frac {19}{768}}\,{t}^{3}+
\frac{1}{6}\,{s}^{3}+
{\frac {5}{1536}}\,{t}^{4}+{\frac {99}{128}}\,s+
\frac{1}{24}\,{s}^{4}+{\frac {1}{120}}\,{s}^{5}-
{\frac {7}{16}}\,t{s}^{2}+
\frac{1}{4}\,s{t}^{2}+
{\frac {249}{512}} \right) {e^{2\,t}} \\ 
& \left. \qquad \quad + {\frac {25}{64}}\,t+{\frac {3}{128}}\,{t}^{3}+
{\frac {1}{512}}\,{t}^{4}+{\frac {263}{512}} +
{\frac {7}{64}}\,ts+{\frac {67}{512}}\,{t}^{2}+
{\frac {29}{128}}\,s+{\frac {1}{64}}\,s{t}^{2}+
\frac{1}{32}\,{s}^{2} 
\right] \\
& + 
Q^3 \left[ 
 \left( -{\frac {9}{512}}\,t{s}^{2}+
\frac{1}{32}\,{s}^{3}- 
{\frac {1}{128}}\,{s}^{3}t+
{\frac {3}{128}}\,ts+{\frac {195}{2048}}\,t+
{\frac {1}{1024}}\,{t}^{4}s-{\frac {21}{2048}}\,{t}^{2}-
{\frac {1}{512}}\,{s}^{2}{t}^{3}+
{\frac {1249}{4096}}+{\frac {511}{1024}}\,s- 
{\frac {3}{1024}}\,{t}^{2}{s}^{2} \right. \right. \\ 
& \left. \qquad \quad -{\frac {9}{512}}\,s{t}^{2}+
{\frac {53}{256}}\,{s}^{2}-{\frac {3}{1024}}\,{t}^{3}+
{\frac {1}{2048}}\,{t}^{4} \right) {e^{t}} \\ 
& \qquad \quad +
 \left( -{\frac {717275}{2239488}}\,t-{\frac {287}{7680}}\,s{t}^{5}+
\frac{1}{12}\,{s}^{4}{t}^{2}+{\frac {1}{720}}\,{s}^{6}-
{\frac {7345}{13824}}\,t{s}^{2}-
{\frac {5933}{414720}}\,{t}^{5}-\frac{1}{24}\,{s}^{5}{t}^{3}+
{\frac {203}{1536}}\,{t}^{4}{s}^{2}+ 
{\frac {1}{3360}}\,{s}^{8} \right.\\ 
& \qquad \quad -{\frac {53}{1280}}\,{s}^{2}{t}^{5}+
{\frac {13}{96}}\,{s}^{3}+{\frac {5}{129024}}\,{t}^{8}-
{\frac {77399}{124416}}\,ts-{\frac {73}{161280}}\,{t}^{7}+
{\frac {194261}{746496}}\,{t}^{2}-{\frac {95}{576}}\,{s}^{3}{t}^{3}
-{\frac {101}{324}}\,{t}^{3}s+{\frac {4591}{9216}}\,{t}^{2}{s}^{2} \\
& \qquad \quad -{\frac {1}{280}}\,{s}^{7}t+{\frac {5}{288}}\,{s}^{6}{t}^{2}-
{\frac {50813}{373248}}\,{t}^{3}+{\frac {21}{2560}}\,{s}^{2}{t}^{6}+
{\frac {179}{23040}}\,s{t}^{6}-{\frac {89}{2880}}\,{s}^{3}{t}^{5}+
{\frac {5}{96}}\,{s}^{4}{t}^{4}-{\frac {5}{5376}}\,s{t}^{7}+
{\frac {1}{80}}\,{s}^{5}{t}^{2} \\ 
& \qquad \quad +{\frac {5154175}{26873856}}-
{\frac {1}{2520}}\,{s}^{7}+
{\frac {139}{46080}}\,{t}^{6}- 
\frac{1}{18}\,{s}^{4}{t}^{3}+
{\frac {1}{720}}\,{s}^{6}t+
{\frac {12689}{248832}}\,{t}^{4}+
{\frac {23}{288}}\,{s}^{3}{t}^{4}+
{\frac {5819}{20736}}\,{s}^{2}-{\frac {1}{60}}\,{s}^{5}t \\
& \left. \qquad \quad 
+ {\frac {258265}{746496}}\,s+{\frac {49}{192}}\,{s}^{3}{t}^{2}-
{\frac {175}{576}}\,{s}^{2}{t}^{3}+
{\frac {3491}{27648}}\,{t}^{4}s-{\frac {101}{384}}\,{s}^{3}t-
\frac{1}{12}\,{s}^{4}t+\frac{1}{24}\,{s}^{4}+{\frac {3787}{6912}}\,s{t}^{2}
+{\frac {1}{120}}\,{s}^{5} \right) {e^{3\,t}} \\
& \left. \qquad \quad + {\frac {1}{162}}\,s{t}^{2} +{\frac {3302}{6561}}+
{\frac {113}{729}}\,s+{\frac {754}{2187}}\,t+{\frac {5}{81}}\,ts+
{\frac {1}{81}}\,{s}^{2}+{\frac {73}{729}}\,{t}^{2}+
{\frac {11}{729}}\,{t}^{3}+{\frac {1}{972}}\,{t}^{4} \right] 
+ O(Q^4). 
\end{align*} 
}
\normalsize 
In the appendix, we give a table of invariants for $\C\Proj^2$. 

\subsection{Finiteness of the small quantum product of $\Fl_3$}
\label{subsec:flag} 
We compute the small quantum product of $\Fl_3$ and 
show that they are finite as power series in the Novikov variables. 
The finiteness of small quantum $K$-ring was first observed 
for Grassmannians \cite{Buch-Mihalcea}. 
The result has been generalized, by Buch, Chaput, Mihalcea and Perrin, 
to cominuscule homogeneous varieties 
\cite{BCMP:cominuscule} 
or more generally to generalized flag varieties of Picard rank one 
whose Gromov-Witten varieties are rationally connected for 
large degrees \cite{BCMP}. 
Our method is purely algebraic and does not rely on 
geometric properties of the moduli spaces (although 
we use the result of Givental-Lee \cite{Givental-Lee}). 
We remark that a complete set of relations 
of the small quantum $K$-ring of $\Fl_n$ has been found 
by Kirillov-Maeno \cite{Kirillov-Maeno}. Their 
relations are polynomials in $Q$ and provide an evidence 
for possible finiteness (but do not imply it). 
A combinatorial model for the small quantum 
$K$-ring of $\Fl_n$ defined over the polynomial ring was 
also proposed by Lenart-Maeno \cite{Lenart-Maeno}. 

We start with a general strategy to prove the finiteness. 
Suppose that $A_{i,\rm com}|_{t=0}$ are polynomials 
in the Novikov variables $Q$, and that we find polynomials 
$f_\beta(x_1,\dots,x_r,Q_1,\dots,Q_r) \in \Q[x_1,\dots,x_r,Q_1,\dots,Q_r]$
such that $f_\beta(A_{1,\rm com}|_{t=0},\dots,A_{r,\rm com}|_{t=0}, 
Q_1,\dots,Q_r) \Phi_0 = \Phi_\beta$. 
Then the small quantum products are polynomials in $Q$ 
because we have 
\begin{equation}
\label{eq:product_intermsof_A} 
\Phi_\alpha \bullet \Phi_\beta \big|_{t=0} 
= f_\beta(A_{1,\rm com}|_{t=0},\dots, A_{r,\rm com}|_{t=0},
Q_1,\dots,Q_r) \Phi_\alpha 
\end{equation} 
by the commutativity of $A_{i, \rm com}|_{t=0}$ and 
$(\Phi_\alpha\bullet)|_{t=0}$. 
 
Let $\Fl_3$ be the manifold of complete flags 
$\{0\}:=V_0\subset V_1\subset V_2\subset V_3:=\C^3$. 
The $K$-ring  $K(\Fl_3)$ is generated by the tautological line bundles 
$L_i := V_i/V_{i-1} (1\leq i\leq 3)$ subject to the following relations 
\begin{align*}
\sum_{i=1}^3(1-L_i) = 0, \quad  
\sum_{1\leq i< j\leq 3} (1-L_i)(1-L_j)=0,\quad  \prod_{i=1}^3 (1-L_i)=0. 
\end{align*}
By the Pl\"{u}cker embedding, $\Fl_3$ can be realized as a 
$(1,1)$-hypersurface in $\Gr(1,3)\times \Gr(2,3) 
\cong \C\Proj^2\times \C\Proj^2$; 
$\Fl_3$ is cut out by the incidence  relation 
``line $\subset$ plane''. 
Let $P_1^{-1}$ and $P_2^{-1}$ be the line bundles on 
$\Fl_3$ that are pullbacks  
of the tautological line bundles on the two copies of $\C\Proj^2$. 
We have
\begin{align*}
L_1=P_1^{-1},\quad L_2 = P_1 P_2^{-1}, \quad L_3= P_2. 
\end{align*}
 The $J$-function of $\Fl_3$ is computed in \cite{Givental-Lee} to be
\begin{align*}
J=(1-q)\sum_{d_1,d_2}\frac{Q_1^{d_1}Q_2^{d_2}
\prod_{r=1}^{d_1+d_2}(1-P_1^{-1}P_2^{-1}q ^r)}
{\prod_{r=1}^{d_1}(1-P_1^{-1}q^r)^3
\prod_{r=1}^{d_2}(1-P_2^{-1}q^r)^3}.
\end{align*}
Set $\tJ = P_1^{-\log Q_1/\log q} P_2^{-\log Q_2/\log q} J$ 
as in Proposition \ref{prop:J_generator}. 
It satisfies the following finite-difference equations:
\begin{align}
\label{eq:diffeq_J_flag3}
\begin{split} 
\left(1-q^{Q_1\partial_{Q_1}}\right)^3 \tJ
& =Q_1 \left(1-  q^{1+ Q_1\partial_{Q_1} +Q_2\partial_{Q_2}} \right ) \tJ, \\
\left(1-q^{Q_2\partial_{Q_2}} \right)^3 \tJ
& =Q_2\left(1-q^{1+Q_1\partial_{Q_1}+ Q_2 \partial_{Q_2}} \right) \tJ.
\end{split}
\end{align}
We choose the ordered basis $\{\Phi_0,\Phi_1,\dots,\Phi_5\}$ of $K(\Fl_3)$ as: 
\[
1,\quad 1-P_1^{-1},\quad (1-P_1^{-1})^2,\quad 1-P_2^{-1},
\quad (1-P_2^{-1})^2,\quad (1-P_1^{-1})(1-P_2^{-1}). 
\]
Setting $\Delta_i := 1 - q^{Q_i\partial_{Q_i}}$, 
one can derive the following identities: 
\begin{align}
\label{eq:diffop_J}
\begin{split} 
\tJ  & =  (1-q) \tT 1 ,\\
\Delta_i \tJ &  
=  (1-q) \tT (1-P_i^{-1}),\quad i=1,2, \\
\Delta_1 \Delta_2 \tJ
& = (1-q) \tT (1-P_1^{-1})(1-P_2^{-1}), \\
\Delta_1^2 \tJ   
& =  (1-q) \tT\left((1-P_1^{-1})^2+ Q_1 P_2^{-1}\right), \\
\Delta_2^2 \tJ  
& =  (1-q) \tT\left((1-P_2^{-1})^2+ Q_2P_1^{-1}\right),\\
\Delta_1^2 \Delta_2 \tJ
& =  (1-q) \tT \left((1-P_1^{-1})^2(1-P_2^{-1}) + 
Q_1P_2^{-1}(1-P_2^{-1}) - Q_1Q_2P_1^{-1}\right), \\ 
\Delta_1 \Delta_2^2 \tJ  
& =  (1-q) \tT \left((1-P_1^{-1})(1-P_2^{-1})^2 
+Q_2P_1^{-1}(1-P_1^{-1}) - Q_1Q_2P_2^{-1}\right),
\end{split}  
\end{align}
where we set $\tT= P_1^{-\log Q_1/\log q} P_2^{-\log Q_2/\log q}T$ 
as in Remark \ref{rem:logQ}. 
We obtain the RHS  by studying the leading term of the expansion at $q=\infty$ 
of the LHS (see the argument in the proof of 
Lemma \ref{lem:initial_condition}). 
The above identities combined with the difference equations 
\eqref{eq:diffeq_J_flag3} for $\tJ$ yield  
the following formulae for the operators $A_1, A_2$ 
(since we consider only small quantum $K$-theory in \S \ref{subsec:flag}, 
we write $A_i$ for $A_i|_{t=0}$): 
 \begin{align*}
 A_1=\begin{pmatrix}
 1&-Q_1&0&0&0&Q_1Q_2\\
 -1&1&0&0&0&-Q_1Q_2\\
 0&-1&1&0&-1&-1\\
 0&Q_1&-Q_1&1&0&-Q_1\\
 0&0&0&0&0&-1+Q_1\\
 0&0&0&-1&1&2
 \end{pmatrix},\\
 A_2=\begin{pmatrix}
 1&0 &0&-Q_2&0&Q_1Q_2\\
 0&1&0&Q_2&-Q_2&-Q_2\\
 0&0&0&0&0&-1+Q_2\\
 -1&0&0&1&0&-Q_1Q_2\\
 0&0&-1&-1&1&-1\\
 0&-1&1&0&0&2
 \end{pmatrix}.
 \end{align*} 
These operators are polynomials in $Q_1,Q_2$ and 
do not depend on $q$. Since we have $A_i 1 = P_i^{-1}$, 
$A_{i} = A_{i,\rm com}$ gives the small quantum 
multiplication by $P_i^{-1}$. 

The first five equations of \eqref{eq:diffop_J} show that 
each basis element $\Phi_\beta$ can be written in the form 
$\Phi_\beta = f_\beta(A_{i},Q_i) \Phi_0$ 
where the polynomials $f_\beta(x_1,x_2,Q_1,Q_2)$ are: 
\[
1, \quad 1-x_1, \quad  (1-x_1)^2 - Q_1 x_2,\quad 
1-x_2, \quad (1-x_2)^2 - Q_2 x_1, \quad (1-x_1)(1-x_2).
\] 
This proves the finiteness of the small quantum product. 
By equation \eqref{eq:product_intermsof_A}, we have
that $(\Phi_\beta \bullet)|_{t=0} = f_\beta(A_1,A_2,Q_1,Q_2)$. 

\begin{remark} 
It was conjectured in \cite{Lenart-Maeno} that the 
quantum Grothendieck polynomials represent Schubert classes 
in the small quantum $K$-ring. When we identify the variables $x_1$, $x_2$ in 
\cite[Example 3.19]{Lenart-Maeno} with the elements 
$1-A_1^{-1}1$ and $1-A_1A_2^{-1}1$, 
we find that the quantum Grothendieck polynomials in 
\cite[Example 3.19]{Lenart-Maeno} represent the corresponding 
Schubert classes except for the last one $\mathfrak{G}_{321}^q$ 
(we see that $\mathfrak{G}_{321}^q + Q_1Q_2$ represents a Schubert 
cycle).  
\end{remark}

\section{Analyticity of Structure Constants} 
In this section we prove two results on the analyticity 
of quantum $K$-theory. 
The first result (Theorem \ref{thm:polynomiality}) 
is about polynomiality in $t$ and $e^t$ of 
the coefficients of $Q^d$ of the structure constants. 
The second result (Theorem \ref{thm:convergence}) 
is about full analyticity of the structure 
constants as functions in $Q$, $t$ and $q$ under 
the Lee-Pandharipande reconstruction 
(Corollary \ref{cor:Lee-Pandharipande}). 

\subsection{Polynomiality in $t$ and $e^{t}$} 
We choose a basis $\Phi_0,\dots,\Phi_N$ of 
$K(X)$ satisfying the following properties:
\begin{itemize} 
\item[(i)] $\Phi_0=1$; 
\item[(ii)] $\ch(\Phi_1),\dots,\ch(\Phi_r)$ are 
cohomology classes of degree $\ge 2$ and 
the degree two components 
of $\ch(\Phi_1),\dots, \ch(\Phi_r)$ form 
a nef integral basis of $H^2(X;\Z)/{\rm torsion}$; 
\item[(iii)] 
$\ch(\Phi_{r+1}),\dots,\ch(\Phi_N)$ 
are cohomology classes of degree $\ge 4$.  
\end{itemize} 

\begin{theorem}
\label{thm:polynomiality} 
For $E_1,\dots, E_m \in K(X)$, non-negative integers 
$k_1,\dots,k_m\ge 0$ and $d\in \Eff(X)$, 
the generating function 
\[
\sum_{n\ge 0}
\frac{1}{n!} 
\corr{E_1L^{k_1},\dots,E_m L^{k_m}, t,\dots,t}_{0,m+n, d}^X 
\]
with $t = \sum_{\alpha=0}^N t^\alpha \Phi_\alpha$ is a 
polynomial in $t^0,\dots,t^N$ and $e^{t^0},\dots, e^{t^r}$. 
\end{theorem} 

\begin{corollary} 
\label{cor:polynomiality} 
Expand the quantum multiplication $\Omega_\alpha = \Phi_\alpha\bullet$ 
and the $q$-shift connection operator 
$A_i = S P_i^{-1} \left( q^{Q_i\partial_{Q_i}} S^{-1}\right)$ 
as power series in $Q$ (see Proposition \ref{prop:A_expansion}) 
\[
\Omega_\alpha  = \sum_{d\in \Eff(X)} \Omega_{\alpha,d}(t) Q^d, 
\qquad 
A_i = P_i^{-1} + \sum_{\substack{d\in \Eff(X)\\ d_i>0}} 
\sum_{k=0}^{d_i-1} A_{i,d,k}(t) Q^d (1-q)^k 
\] 
where $d_i = \pair{p_i}{d}$. 
Then the coefficients $\Omega_{\alpha,d}(t)$ and $A_{i,d,k}(t)$ are polynomials 
in $t^1,\dots,t^N$ and $e^{t^1},\dots,e^{t^r}$. 
\end{corollary} 

The proof of Theorem \ref{thm:polynomiality} relies on twisted cohomological
GW theory. Let $\psi_i=c_1(L_i)$ and let $\pi:X_{0,n+1,d}\to
X_{0,n,d}$ be the universal family. Recall that a multiplicative characteristic
class is by definition a function of the type $E\mapsto
e^{\sum_{k=0}^\infty s_k\ch_k(E)}$, where $s_k\in \C$, $k=0,1,2,\dots$ 
are some constants. For any topological
space $X$, the above formula defines a map from $K(X)\to
H^*(X)$. Let $\mathcal{A},\mathcal{B},\mathcal{C}$ be a triple of
multiplicative characteristic classes. We will make use of {\em twisted} 
GW invariants of the form
\begin{align}\label{ABC-tw}
\la v_1(\psi), \ldots , v_n({\psi}) , \mathcal{A}(\pi_* \ev_{n+1}^* E), 
\mathcal{B}(\pi_*(1-L^{-1}_{n+1})),
\mathcal{C}((-\pi_*i_*\mathcal{O}_\mathcal{Z})^\vee\ra_{0,n,d},
\end{align}
where $E\in K(X)$,  $v_i(z)\in H^*(X)[\![z]\!]$ , $\mathcal{Z}$ is the
codimension two nodal locus in the universal family. 
We are using the same correlator notation for
twisted cohomological GW invariants as for the $K$-theoretic ones and
the last 3 slots of the correlator are cohomology classes on
$X_{0,n,d}$ that should be cupped with the integrand of the usual GW
invariants. 

The invariants \eqref{ABC-tw} have been studied in \cite{to1}, 
where the effect of each type of
twisting on the generating series of GW theory is described. Let us
briefly review the results that we need and refer to \cite{to1} for
more details and proofs. Put
  \begin{align*}
  \mathcal{F}_{\mathcal{A},\mathcal{B},\mathcal{C}}^g(\mathbf{t}(z))
  =\sum_{n,d}\frac{Q^d}{n!}\la \mathbf{t}(\psi),\ldots
  ,\mathbf{t}(\psi),\mathcal{A}(\pi_*\ev_{n+1}^* E),
  \mathcal{B}(\pi_*(1-L^{-1}_{n+1})),
\mathcal{C}((-\pi_*i_*\mathcal{O}_\mathcal{Z})^\vee)\ra_{g,n,d}
\end{align*}
for the twisted genus-$g$ potential and
\begin{align*}
\mathcal{D}_{\mathcal{A},\mathcal{B},\mathcal{C}}(\hbar,\mathbf{t}(z))=\exp\left( \sum_{g\geq 0} \hbar^{g-1}\mathcal{F}_{\mathcal{A},\mathcal{B},\mathcal{C}}^g(\mathbf{t}(z)) \right)
\end{align*}
for the twisted total descendant potential. The argument $\mathbf{t}(z)$ should be understood as a formal power series in $z$ with coefficients in $H^*(X)$.
 We will also denote by $\mathcal{D}_{\mathcal{A},\mathcal{B}}, \mathcal{D}_{\mathcal{A}}$ etc. potentials which are twisted only by the corresponding types characteristic classes. 
Twisting by the class $\mathcal{B}(\pi_* (L_{n+1}^{-1}-1))$ is tantamount to the translation
\begin{align}
\mathcal{D}_{\mathcal{A,B}}(\hbar,\mathbf{t})
=\mathcal{D}_{\mathcal{A}}\left(\hbar,  \mathbf{t} +z - z\mathcal{B}\left(
\frac{\mathbf{L}_{z}^{-1}-1}{\mathbf{L}_{z}-1}\right)\right)
\cdot K_{\mathcal{B}} , \label{trans} 
\end{align}
where $K_\mathcal{B}$ is a constant, whose precise value will be
irrelevant for what follows, so we may assume that it is 1.
   
The class $\mathcal{C}((-\pi_* i_* \mathcal{O}_\mathcal{Z})^\vee)$ is
localized near the  nodes. A correlator cupped with such
a characteristic class is a sum of correlators corresponding to the
stratification of the codimension-2 singular locus
$\mathcal{Z}\subset X_{0,n+1,d}$ by the number of nodes. Each strata is a product of
moduli spaces and each twisted correlator corresponding to a given
strata factors respectively as a product of
correlators (with the $\mathcal{C}$-twist removed). More precisely, we have
     \begin{align}
      \mathcal{D}_{\mathcal{A,B,C}} = \exp\left(\frac{\hbar}{2}\sum_{a,b,\alpha,\beta}A_{a,\alpha;b,\beta}\partial_{a}^{\alpha}\partial_b^{\beta} \right)\mathcal{D}_{\mathcal{A,B}}, \label{psq}
    \end{align}   
where the sum is over all $ a,b\geq 0$, $\{\varphi_\alpha\}$ is a
fixed basis of $H^*(X;\C)$, and $\partial_a^\alpha$ (resp. $\partial_b^\beta$) is
the derivative in the direction of vector $\varphi_\alpha z^a$ (resp. $\varphi_\beta z^b$). The
coefficients $A_{a,\alpha;b,\beta}$ depend on the 
characteristic class $\mathcal{C}$. Their precise value will not be
important, but we have to use the following property: 
\begin{align}\label{dim-constr}
A_{a,\alpha;b,\beta}\neq 0 \quad \Rightarrow \quad 
\deg(\varphi_\alpha)+\deg(\varphi_\beta)=\dim (X).
\end{align}
Let now $p\in H^2(X)$ and denote
  \begin{align*}
  \delta_0(v(z))=v(z),\quad \delta_1 (v(z)) = p\smile\left[\frac{v(z)}{z}\right]_+ ,\quad \delta_{n+1}=\delta_1(\delta_n(v(z)).
   \end{align*}
 where $[\cdots]_+$ denotes the power series truncation. Notice that
   $\pi^*(\pi_* \ev_{m+1}^*E) =\pi_*\ev_{m+2}^*E,$ so the presence of
   the characteristic class $\mathcal{A}$ in the integrand does not
   change the divisor equation, i.e., 
   \begin{align}
 \label{div00} & \la v_1(\psi) ,\ldots ,v_m(\psi),p,
 \mathcal{A}(\pi_*\ev_{m+2}^* E)\ra_{0,m+1,d} = \\ \nonumber
&\langle p,d \rangle \la v_1(\psi),\ldots v_m(\psi), \mathcal{A}\ra_{0,m,d}
   +\sum_{j=1}^m \la v_1(\psi),\ldots, \delta_1v_j(\psi),\ldots ,v_m({\psi}), \mathcal{A}\ra_{0,m,d}. 
   \end{align} 
where on the RHS we denoted $\mathcal{A}=\mathcal{A}(\pi_*\ev^*_{m+1}(E))$. 
  We now prove the following 
  \begin{lemma}
  \label{poly00}
   Let $m$ be a positive integer and $t$ the coordinate dual to $p\in H^2(X)$. Then the series
  \begin{align}
 \label{div02} \sum_{n\geq 0} \frac{t^n}{n!}\la p,\ldots ,p, v_1(\psi), \ldots ,v_m(\psi), \mathcal{A}(\pi_*(\ev^*E))\ra_{0,n+m,d}
  \end{align}
  is a polynomial in $t$ and $e^t$.
  \end{lemma}
  \begin{proof}  First notice we can assume that $v_i(z)$ are polynomials in $z$ : 
namely if  $D=\dim X_{0,m,d}$ then only the degree $D$ truncation of each $v_i(z)$ (where $\deg(z)=1$) gives non-trivial contributions in the correlators, as can be checked by an easy dimensional count. 

 Now setting $\lambda=\langle p,d\rangle$, we apply the divisor equation (\ref{div00})  repeatedly for a fixed $n$:
  \begin{align*}
  &\frac{t^n}{n!}\la v_1(\psi), \ldots, v_m(\psi),p,\ldots ,p ,\mathcal{A}(\pi_*(\ev^*E))\ra_{0,n+m,d} =  \\
 &=\frac{t^n \lambda}{n!}\la v_1,\ldots ,v_m,p, \ldots p, \mathcal{A}\ra_{0,n+m-1,d}
 +\frac{t^n}{n!}\sum^m_{i=1}\la v_1,\ldots ,\delta_1(v_i),\ldots, p,\ldots p, \mathcal{A}\ra_{0,n+m-1,d} = \ldots \\
 &=\frac{t^n \lambda^n}{n!}\la v_1,\ldots ,v_m, \mathcal{A}\ra_{0,m,d} + \sum_{k\geq 1}\sum_{j_1 + \ldots +j_m =k}\frac{t^n \lambda^{n-k}  \binom{n} {k}}{n!}\la \delta_{j_1}(v_1),\ldots ,\delta_{j_m}(v_m),\mathcal{A}\ra_{0,m,d} .
  \end{align*}
  Summing after $n$ we see that the series (\ref{div02}) equals
  \begin{align*}
   & \sum_{k\geq 0}\sum_{j_1+\ldots + j_m=k}\la \delta_{j_1}(v_1),\ldots ,\delta_{j_m}(v_m),\mathcal{A}\ra_{0,m,d}\left( \sum_{n\geq k}\frac{t^n \lambda^{n-k}\binom{n} {k}}{n!}\right)= \\
  &= \sum_{k\geq 0}\sum_{j_1+\ldots + j_m=k}\la \delta_{j_1}(v_1),\ldots ,\delta_{j_m}(v_m),\mathcal{A}\ra_{0,m,d} \frac{t^k e^{\lambda t}}{k!}.
    \end{align*}
The polynomiality of $v_i$ ensures that there are finitely many $j_i$ such that $\delta_{j_i}(v_i)\neq 0$, hence the above sum is finite in $k$. This concludes the proof of the Lemma. 
\end{proof}
 \begin{proof}[{{\bf Proof of Theorem \ref{thm:polynomiality}}}] 
The polynomiality on $t^0, e^{t^0}$ can be seen directly 
from the $K$-theoretic string equation, so henceforth we assume $t^0=0$. 
 
We use Kawasaki-Riemann-Roch (KRR) theorem of \cite{kaw}, which states
that the Euler characteristics of an (orbi)bundle on a smooth compact
orbifold $\mathcal{Y}$  equals a certain cohomological integral on its
inertia orbifold $I\mathcal{Y}$. The latter is defined set-theoretically as pairs $(y, (g))$ where $y\in \mathcal{Y}$ and $(g)$ is a (conjugacy class of) a symmetry fixing $y$.  It is a smooth (generally disconnected) orbifold.  The connected component of $I\mathcal{Y}$ associated with identity is isomorphic to $\mathcal{Y}$ and the corresponding term in KRR theorem 
 is the right hand side of the Grothendieck-Riemann-Roch theorem. We will not need the exact statement of KRR, suffices to say it takes the form
 \begin{align*}
 \chi(\mathcal{Y},V)=\int_{\mathcal{Y}}\ch(V)
\Td(T_\mathcal{Y})+\int_{\mathcal{Y}_\mu}\cdots ,
 \end{align*}
 where $T_\mathcal{Y}$ is the tangent bundle to $\mathcal{Y}$ and
 $\mathcal{Y}_\mu$ are connected components of 
 $I\mathcal{Y}$ corresponding to non-trivial symmetries,  also known
 as {\em Kawasaki strata}. According to \cite{to2}, the KRR formula
 can be applied to the moduli space $X_{0,n+m,d}$, i.e., 
  \begin{align}
  \label{krr0}
   & \la E_1 L^{k_1},\ldots, E_m L^{k_m} , t,\ldots ,t \ra_{0,m+n,d} =
   \\ \nonumber
&
\la \ch(E_1 L^{k_1}),\ldots ,\ch(E_m L^{k_m}) , \ch(t),\ldots,\ch(t) ,\Td(\mathcal{T}_{n+m,d})\ra_{0,m+n,d} + \cdots
   \end{align}
  where the correlator on the RHS is cohomological and $\mathcal{T}_{n+m,d}$ is the virtual  tangent bundle of the moduli space. 
Let us ignore for now the dots on RHS of \eqref{krr0} and prove that for a given $m\geq 0$ and $v_1(z), \ldots , v_m(z)\in H^*(X)[\![z]\!]$ the series
   \begin{align}   
  \label{div04}   \sum_{n}\frac{1}{n!}\la v_1(\psi),\ldots ,v_m(\psi), 
\ch(t), \ldots , \ch(t), \Td(\mathcal{T}_{n+m,d})\ra_{0,n+m,d}
   \end{align}
is polynomial in $t^i$, $i=1,\dots,N$ and 
$e^{t^i}$, $i=1,\dots,r$. The polynomiality in coordinates 
  $t^i$, $i\geq r+1$ is easy to see for dimensional reasons: an input with 
$\ch(\Phi_i)\in H^{\geq 4}(X)$ adds one to the dimension of the moduli space 
and two to the cohomological degree of the integrand, 
hence only finitely many monomials in these coordinates give non-zero contributions.
It therefore suffices to assume $\ch(t)=t^1p\in H^2(X)$ 
in the series (\ref{div04}). 
The tangent bundle $\mathcal{T}_{n+m,d}$ can be written as 
an element in  $K(X_{0,n+m,d})$  (see \cite{ctom}, \cite{to1})
\begin{align}
\label{tang00}   \mathcal{T}_{n+m,d} =\pi_*(\ev^* T_X-1) 
- \pi_*(L_{n+m+1}^{-1}-1)-(\pi_*i_*\mathcal{O}_\mathcal{Z})^\vee .
\end{align}
Taking $\mathcal{B}=\Td$ in \eqref{trans} we see that twisting by
 $\Td(\pi_*(L_{n+m+1}^{-1}-1))$ acts on the potential
 $\mathcal{D}_\mathcal{A}(\hbar,{\bf t})$ and its partial derivatives (we
 ignore the constant $K_\mathcal{B}$) by the translation 
    \begin{align}
    {\bf t}\mapsto {\bf t}+ z-z\Td\left( \frac{L_z^{-1}-1}{L_z-1} \right) 
= {\bf t}+z+1-e^z =t- \frac{z^2}{2}-\frac{z^3}{3!}-\cdots := t+v(z).
    \end{align}
    Extracting the coefficient of $Q^d$ in genus $0$ (in formula
    \eqref{trans}) we see that the correlator series obtained from
    \eqref{div04} by suppressing the $\mathcal{C}$-twisting (i.e. the
    last term in the tangent bundle \eqref{tang00}) is (using also
    that $\ch(t)=t^1\,p$)
    \begin{align*}
   \sum_{n}&\frac{1}{n!}\la v_1(\psi),\ldots ,v_m(\psi),\ch(t)+
   v(\psi), \ldots , \ch(t)+v(\psi), \Td(\pi_*(\ev^*
   T_X-1))\ra_{0,m+n,d} =\\ 
    = &\sum_{n,k}\frac{(t^1)^n}{n!k!}\la v_1(\psi),\ldots ,v_m(\psi), 
p, \ldots , p,v(\psi),\ldots ,v(\psi), \Td(\pi_*(\ev^* T_X-1))
\ra_{0,m+n+k,d},
\end{align*} 
where on the RHS $n$ and $k$ are respectively the numbers of 
$p$- and $v(\psi)$-insertions. Since the translation vector $v(z)\in z^2
H^*(X)[z]$, the above sum is finite in $k$ for fixed $m$, $n$, and $d$.
Moreover, for dimensional reasons we may assume that $v(z)$ is a polynomial in $z$.
    
According to \eqref{psq} the correlator series \eqref{div04} can be
expressed as a Feynman-type sum over connected decorated trees $T.$ The
vertices $V(T)$ are colored by the parts of a partition of the insertions
$\{v_1(\psi),\dots,v_m(\psi),\ch(t),\dots,\ch(t)\}$ ($n$
repetitions of $\ch(t)$) as well as with degrees $d_v$, s.t., $\sum
d_v=d$, and the flags of $T$ (a flag is a pair of a vertex and an incident
edge) are colored by $(a,\alpha)$ (see formula \eqref{psq}). 
The correlator series \eqref{div04} takes the form
\begin{align}\label{tree-sum}
\sum_T \frac{1}{|\Aut(T)|}\, \prod_{e\in E(T)} A_e \prod_{v\in
  V(T)} \la
\dots,\overbrace{\ch(t),\dots,\ch(t)}^{n_v},\dots,\mathcal{A},\mathcal{B}
\ra_{0,m_v+n_v+l_v,d_v} 
\end{align}
where $E(T)$ is the set of edges of $T$, $A_e:=A_{a,\alpha;b,\beta}$ 
with $(a,\alpha)$ and $(b,\beta)$ 
being the labels corresponding to the two flags of the edge $e$. The
insertions of the correlators are determined by the labels of the
vertices and the flags, so that the first $m_v$ insertions come from
$\{v_1(\psi),\dots,v_m(\psi)\}$, the next $n_v$ insertions are
$\ch(t)$, and the last $l_v$ insertions have the form $\varphi_\alpha \psi^a$
and correspond to the labels of the flags incident with the vertex $v$.  

We have to prove that the sum \eqref{tree-sum} is polynomial in $t^1$
and $e^{t^1}$ if
$m=\sum_v m_v$ and $d=\sum_v d_v$ are fixed. Arguing as above, since
the $\mathcal{B}$-twisting amounts to a translation that does not
change the polynomiality property, we may assume that
$\mathcal{B}=1$. 
We claim that the number of vertices of $T$ is bounded by a number that
depends only on $d$ and $m$. 
Consider first the case $d=0$, i.e., $d_v=0$ for all $v$. 
By dimensional reasons the total number of $\ch(t)$-insertions is
$\sum_v n_v\leq \dim(X)$. Indeed, the correlator corresponding to a
vertex $v$ is an integral over
$\overline{\mathcal{M}}_{0,m_v+n_v+l_v}\times X$, so the total degree
of classes from $H^*(X)$ that are inserted in the correlator is at
most $\dim(X)$. Using the dimension constraint
\eqref{dim-constr} we get 
\begin{align*}
\sum_v n_v + |E(T)|\, \dim(X) \leq
|V(T)|\,\dim(X)
\end{align*}
It remains only to use that $|E(T)|= |V(T)|-1$. Finally, 
by stability, each correlator has at least $3$ insertions, so  
\begin{align*}
3|V(T)| \leq \sum_v (m_v+n_v+l_v) = m+ \Big( \sum_v n_v\Big) +
2|E(T)|\leq  m+ \dim(X)+ 2|V(t)| -2,
\end{align*} 
 which gives the desired bound for $|V(T)|$.
If the total degree is $d>0$, then the total number of vertices $v$ of
$T$, s.t.,  $d_v\neq 0$ is bounded by $\la\omega,d\ra$ where $\omega$
is a very ample class. Removing from $T$ all vertices of non-zero
degree splits the tree into several trees of degree 0. The number of
such trees is at most $\la \omega,d\ra + 1$, while the number of vertices in each
of them can be bounded as it was explained above. This proves
that the number of vertices of $T$ is uniformly bounded.

Let us fix the first $m_v$ and the last $l_v$ insertions in each
correlator in \eqref{tree-sum}. This produces finitely many
configurations, because the first $m_v$ insertions correspond to a partition of a
finite set, hence there are finitely many possibilities, while for the
last $l_v$ insertions, note that the total number of insertions
$\sum_v l_v=2|E(T)|$ is bounded and for dimensional reason we can
insert only finitely many classes of the type 
$\alpha \psi^a$. It follows that for fixed $m$ and $d$, the sum
\eqref{tree-sum} is a polynomial expression of correlator sums of the
type \eqref{div02}, so according to Lemma \ref{poly00}, it 
is polynomial in $t^1$ and $e^{t^1}$. 
      
Going back to \eqref{krr0}: the dots on the RHS are integrals over 
the non-trivial Kawasaki strata of $X_{0,n+m,d}$.  
They can be described combinatorially as follows. 
The points of $X_{0,n+m,d}$ with non-trivial symmetries represent maps 
whose domain has one irreducible component $C$ 
such that the map factors $C\to C' \to X$ 
where the first map is $z\mapsto z^m$ for some $m\geq 2$. 
Notice the only marked points lying on $C$ can be $0$ and $\infty$ 
since they have to be fixed by any symmetry. 
We refer to \cite{Givental-Tonita} for a detailed discussion of these strata. 
It is clear from this description that for a fixed $d$, the number of such 
components $C$ which are multiple covers (and hence the number 
of marked points lying on them) is bounded: hence the contributions 
from integrals over these strata in the generating series is always 
polynomial in $t^i$. Of course such strata can still have irreducible 
components without symmetries ---  and the generating series is 
of the form \eqref{div04} for a  degree less than $d$. 
Now an inductive argument finishes the proof. 
For $d=0$ the moduli spaces are manifolds --- this ensures 
the initial step of the induction is in order.
\end{proof}

\subsection{Preliminary results on the convergence of 
$\cA_i$, $S$ and $T$} 

\begin{lemma} 
\label{lem:extend_domain}
Let $1\le i\le r$ be an integer and 
let $f= f(q,Q_1,\dots,Q_r)$ be a power series of the form 
\[
f= \sum_{d =(d_1,\dots,d_r) \in (\Z_{\ge 0})^r}  
\sum_{k=0}^{d_i} f_{d,k}q^k Q_1^{d_1} \cdots Q_r^{d_r}.  
\] 
Suppose that $f$ is convergent as a power series in 
$q,Q_1,\dots,Q_r$. 
Then there exists a positive constant 
$\rho>0$ such that $f$ is convergent and bounded on 
the set 
\begin{equation} 
\label{eq:daikei} 
\{(q,Q_1,\dots,Q_r) : |qQ_i|\le \rho, |Q_j|\le \rho \ (\forall j)\}. 
\end{equation} 
\end{lemma}
\begin{proof} 
The assumption implies that there exist constants $C, \epsilon>0$ such 
that $|f_{d,k}|\le C \epsilon^{-d_1-\cdots -d_r-k}$. 
Choose $0<\delta<\min(1,\epsilon)$ and set $\rho = \delta\epsilon$. 
For $(q,Q_1,\dots,Q_r)$ with $|qQ_i|\le \rho$ and $|Q_j|\le \rho$ 
for all $j$, we have 
\begin{align*} 
\sum_{d,k\ge 0} |f_{d,k}q^k Q_1^{d_1}\cdots Q_r^{d_r}| 
& \le 
\sum_{d,k\ge 0} |f_{d + ke_i, k} (qQ_i)^k Q_1^{d_1} 
\cdots Q_r^{d_r}|  \\ 
& \le 
\sum_{d,k\ge 0} C \epsilon^{-2k-d_1- \cdots -d_r} \rho^{k+d_1+\cdots +d_r} \\
& = C \sum_{d,k\ge 0} (\delta/\epsilon)^k \delta^{d_1+\cdots+d_r} 
=\frac{C}{(1-\delta/\epsilon)(1-\delta)^r}
\end{align*} 
where $d =(d_1,\dots,d_r) \ge 0$ means that $d_j \ge 0$ for all $j$ 
and we set $d + ke_i = (d_1,\dots, d_{i-1}, d_i + k, d_{i+1},\dots,d_r)$. 
The conclusion follows. 
\end{proof} 

In view of Proposition \ref{prop:A_expansion}, 
a natural domain of convergence for the $q$-shift operator  
$\cA_i$ should be of the form \eqref{eq:daikei}. 
We now study the relationship between the analyticity of the $q$-shift operators 
$\cA_i$ and that of the fundamental solutions $S$, $T$. 

\begin{proposition} 
\label{prop:convergence_S}
Suppose that the $q$-shift operators $\cA_i = A_i q^{Q_i\partial_{Q_i}}$, 
$1\le i\le r$ are convergent as power series in $q$, $Q_1,\dots,Q_r$, 
$t^0,\dots, t^N$. 
Then there exists a positive number $\rho>0$ such that the 
fundamental solutions $S, T$ in 
Theorem \ref{thm:fund_sol} 
are convergent in the domain 
\[
\{(q,Q_1,\dots,Q_r, t^0,\dots,t^N)  : 
|q|\neq 1,  |Q_i|<\rho, |t^\alpha|<\rho \}.  
\]
In particular, the $J$-function is convergent on the same region. 
\end{proposition} 
\begin{proof} 
By the same argument as Lemma \ref{lem:extend_domain}, 
under the assumption in the proposition, 
we can show that $A_i$ is convergent and bounded on 
the domain of the form 
\[
\{(q,Q_1,\dots,Q_r, t^0,\dots,t^N) : |qQ_i|\le \epsilon, |Q_j|\le \epsilon 
\ (\forall j), |t^\alpha|\le \epsilon \ (\forall \alpha) \} 
\]
for some $\epsilon>0$ 
(this is a parametric version of Lemma \ref{lem:extend_domain}). 
We expand 
\[
A_i = \sum_{n=0}^\infty A_{i,n} Q_i^n 
\]
where $A_{i,n}$ is a function of $q$, $\{Q_j\}_{j\neq i}$ and 
$t^0,\dots,t^N$. 
Fix a norm on $K(X)_\C = K(X)\otimes \C $ and consider 
the corresponding operator norm $\|\cdot\|_{\rm op}$ on $\End(K(X)_\C)$. 
Using Cauchy's integral formula we obtain the estimate 
\begin{equation} 
\label{eq:estimate_Ain}
\|A_{i,n}\|_{\rm op} = \frac{1}{2\pi}\,\left \|
\int_{|Q_i|= \min(\epsilon, \epsilon/|q|)} A_i (q,Q,t) \frac{dQ_i}{Q_i^{n+1}} 
\right \|_{\rm op} 
\le C_1 \epsilon^{-n} \max(|q|^n,1) 
\end{equation} 
for all $q$, $\{Q_j\}_{j\neq i}$, $t^0,\dots,t^N$ with 
$|Q_j|\le \epsilon$ and $|t^\alpha|\le \epsilon$. 
Here $C_1>0$ is a constant.

Since $S= T^{-1}$ (Proposition \ref{prop:unitarity}), 
it suffices to prove the convergence of $T$. 
We assume by induction that we know the convergence of 
$T$ on the domain: 
\[
\left\{(q,Q_1,\dots,Q_r, t^0,\dots, t^N) : |q|\neq 1, \ 
|Q_j|<\rho \ (j\le i-1), Q_j=0 \ (j\ge i), 
|t^\alpha|<\rho \ (\forall \alpha) \right\} 
\]
for some $1\le i\le r$ and $0<\rho<\epsilon$. 
The induction hypothesis holds for $i=1$ 
because $T|_{Q=0} = 
\exp(\frac{1}{1-q} (t\otimes))$. 
(This follows, for example, from the differential equation 
in Theorem \ref{thm:fund_sol}). 
We show the convergence of $T$ in the direction 
of $Q_i$ using the $q$-difference equation (see \eqref{eq:def_A}): 
\[
P_i^{-1} q^{Q_i \partial_{Q_i}} T = T A_i. 
\]
Expand 
\begin{align*} 
T|_{Q_{i+1}=\cdots = Q_r=0} &= \sum_{n=0}^\infty T_n Q_i^n. 
\end{align*} 
The coefficient $T_n$ is a function of $q, Q_1,\dots,Q_{i-1}, t^0,\dots, 
t^N$. The difference equation implies that: 
\begin{equation}
\label{eq:recursion_Tn}
q^n T_n - P_i T_n P_i^{-1} = P_i( T_{n-1} A_{i,1} + \cdots + T_0 A_{i,n} ). 
\end{equation} 
This equation determines $T_n$ recursively because 
the operation 
$M \mapsto (q^n - \Ad(P_i))(M) = q^n M - P_i M P_i^{-1}$ is invertible 
for $|q|\neq 1$. 
We need to show that the sum $\sum_{n=0}^\infty T_n Q_i^n$ 
converges for small $|Q_i|$. 

First we consider the case where $|q|\le \delta$ for some $\delta<1$. 
Choose sufficiently big $n_0\gg 1$ such that $\delta^{n_0}<1/2$. 
For $n\ge n_0$, since $P_i$ is unipotent, we have an estimate 
\[
\|(q^n - \Ad(P_i)) M \|_{\rm op} \ge C_2 \|M\|_{\rm op}
\]
for a constant $C_2>0$ independent of $n\ge n_0$ and $\delta$. 
Therefore we have for $n\ge n_0$, 
\[
\|T_n\|_{\rm op} \le C_2^{-1} \|P_i\|_{\rm op} 
\sum_{k=1}^n \|T_{n-k}\|_{\rm op} 
\|A_{i,k}\|_{\rm op} 
\le C_1C_2^{-1} \|P_i\|_{\rm op} 
\sum_{k=1}^n \|T_{n-k} \|_{\rm op} \epsilon^{-k}  
\]
where we used \eqref{eq:estimate_Ain}. 
Set $C_3 = C_1C_2^{-1} \|P_i\|_{\rm op}$. 
Choose a constant $C_4>0$ such that $\|T_k\|_{\rm op} 
\le C_4$ for all $0\le k< n_0$ and all 
$(q,Q_1,\dots,Q_{i-1},t^0,\dots,t^N)$ 
such that $|q|\le \delta$, $|Q_j|\le \rho/2$ ($\forall j$), 
$|t^\alpha|\le \rho/2$ ($\forall \alpha$). 
(Note that $C_4$ can depend on $\delta$.) 
Also choose 
$C_5>0$ such that $C_5 \ge \max(1, 2 \epsilon^{-1}, 
2C_3 \epsilon^{-1})$. 
Note that $C_5$ does not depend on $\delta$. 
We show that $\|T_n\|\le C_4 C_5^n$ holds for all $n\ge 0$ 
and all $(q,Q_1,\dots,Q_{i-1}, t^0,\dots,t^N)$ 
with $|q|\le \delta$, $|Q_j|\le \rho/2$ ($\forall j$), 
$|t^\alpha|\le \rho/2$ ($\forall \alpha$). 
We already know that this holds for $0\le n< n_0$. Suppose by induction that 
we have $\|T_k\|\le C_4C_5^k$ for all $0\le k < n$ 
for some $n\ge n_0$. Then we have: 
\[
\|T_n\|_{\rm op} \le C_3 \sum_{k=1}^{n} C_4 C_5^{n-k} \epsilon^{-k} 
=C_4 \frac{C_3}{\epsilon C_5}  C_5^n \sum_{l=0}^{n-1} \frac{1}{(\epsilon C_5)^l} 
\le C_4 \frac{2C_3}{\epsilon C_5} C_5^n 
\le C_4 C_5^n. 
\]
This completes the induction step: we know that 
the series $\sum_{n=0}^\infty T_n Q_i^n$ 
converges on the set 
\[
\{(q,Q_1,\dots,Q_r, t^0,\dots,t^N): |q|< \delta, |Q_j|<
\sigma \ (j\le i), Q_j = 0\ (j > i), |t^\alpha| < \sigma \ (\forall \alpha)\} 
\]
where $\sigma = \min(\rho/2, C_5^{-1})$. Note that $\sigma$ does 
not depend on $\delta<1$. Since $\delta<1$ was arbitrary, 
we can replace the inequality $|q|<\delta$ above with $|q|<1$. 

Next we consider the case where $|q|\ge \delta$ for some 
$\delta>1$. 
The discussion in this case is similar except that 
we divide both sides of \eqref{eq:recursion_Tn} by $q^n$. 
We choose $n_0'\gg 0 $ such that $\delta^{-n_0'}<1/2$. 
For $n\ge n_0'$, we have an estimate: 
\[
\|(1 - q^{-n} \Ad(P_i)) M \|_{\rm op} \ge  C_2' \|M\|_{\rm op} 
\]
for some $C_2'>0$ independent of $n$ and $\delta$. 
Using \eqref{eq:recursion_Tn} and \eqref{eq:estimate_Ain} again, 
we have for $n\ge n_0'$  
\[
\|T_n\|_{\rm op} \le C_1 {C_2'}^{-1}\|P_i\|_{\rm op} 
\sum_{k=1}^n |q|^{-(n-k)} \|T_{n-k}\|_{\rm op} \epsilon^{-k} 
\le C_3' \sum_{k=1}^n \|T_{n-k}\|_{\rm op} \epsilon^{-k}. 
\]
where $C_3' = C_1 {C_2'}^{-1} \|P_i\|_{\rm op}$. 
By the same discussion as above, we can show that 
there exist $C_4', C_5'>0$ such that we have 
$\|T_n\|_{\rm op}\le C_4' {C_5'}^n$ 
for all $n\ge 0$ and all $(q,Q_1,\dots,Q_{i-1},t^0,\dots,t^N)$ 
with $|q|\ge \delta$, $|Q_j|\le \rho/2$ ($\forall j$) and 
$|t^\alpha|\le \rho/2$ ($\forall \alpha$). 
(Here recall that $T|_{q=\infty} = \id$.)
Also $C_5'$ is independent of $\delta>1$. 
Thus we know that $\sum_{n=0}^\infty T_n Q_i^n$ converges 
also on the set: 
\[
\{(q,Q_1,\dots,Q_r, t^0,\dots,t^N) : |q|>1, |Q_j|< \sigma' 
\ (j\le i), Q_j=0 \ (j>i), |t^\alpha|< \sigma' \ (\forall \alpha)\} 
\]
for $\sigma' = \min(\rho/2, {C_5'}^{-1})$. 
This completes the induction on $i$, and the proof of 
the proposition. 
\end{proof}

\begin{remark} 
In the above proposition, the convergence assumption 
in the direction of $t$ can be weakened as follows. 
If we assume the convergence of $\cA_i$, $i=1,\dots,r$ 
with $t=(t^0,\dots,t^N)$ restricted to 
a certain submanifold $D$ of $\C^{N+1}$, 
then the fundamental solutions $S$, $T$ converge 
in the domain of the form 
\[
\{(q,Q_1,\dots,Q_r, t^0,\dots,t^N) \in \C^{r+1}\times D 
: |q|\neq 1,  |Q_i|<\rho, |t^\alpha|<\rho \}.  
\]
In particular, the convergence of $\cA_i|_{t=0}$ implies 
the convergence of $S|_{t=0}$, $T|_{t=0}$. 
\end{remark} 
 
\subsection{Analyticity of the Reconstruction} 
In this section we prove that our Reconstruction Theorem 
\ref{thm:reconstruction} preserves the analyticity 
under the assumption that $K(X)$ is generated by line 
bundles. 
\begin{theorem}
\label{thm:convergence} 
Suppose that $K(X)$ is generated by line bundles as a ring. 
Suppose also 
that the small $q$-shift operators $\cA_i|_{t=0}$, $i=1,\dots,r$ at $t=0$ 
are convergent as power series in $q$ and $Q_1,\dots,Q_r$. 
Then there exists $\epsilon>0$ such that 
\begin{itemize} 
\item[(1)] 
the GW potential $\cF$, the metric $G$ 
and the big quantum products $\Omega_\alpha = (\Phi_\alpha\bullet)$ 
are analytic on  
$\{(Q_1,\dots,Q_r, t^0,\dots,t^N) : 
|Q_j|<\epsilon \ (\forall j), \ |t^\beta|<\epsilon 
\ (\forall \beta) \}$; 

\item[(2)] the big $q$-shift operator $\cA_i = A_i q^{Q_i\partial_{Q_i}}$ 
is analytic on  
$\{(q, Q_1,\dots,Q_r, t^0,\dots,t^N) : 
|qQ_i|<\epsilon, |Q_j|<\epsilon \ (\forall j), 
|t^\alpha|<\epsilon\ (\forall \alpha) \}$ 
(note that the domain depends on $i$); 

\item[(3)] the fundamental solutions $S$, $T$ are analytic on 
$\{(q,Q_1,\dots,Q_r, t^0,\dots,t^N) : 
|q|\neq 1, |Q_j|<\epsilon \ (\forall j), 
|t^\alpha|<\epsilon \ (\forall \alpha) \}$. 
\end{itemize} 
\end{theorem}

The rest of the section is devoted to the proof of 
Theorem \ref{thm:convergence}. 
The proof is based on a version of
``abstract Cauchy-Kowalevski theorem" due to Nishida 
\cite{Nishida}. 
We start with a review of the statement.  
(Compared to the original paper, we shifted the origin of 
$u$ and $\rho$ for our convenience.) 
Let $\{B_\rho\}_{\rho\ge \epsilon}$ be a family of Banach spaces 
such that $B_\rho \subset B_{\rho'}$ and 
$\|\cdot\|_{\rho'} \le \|\cdot\|_{\rho}$ for all $\epsilon \le \rho'<\rho$. 
Consider the ordinary differential equation with values in $B_\rho$: 
\begin{equation} 
\label{eq:diffeq}
\frac{du}{dt} = F(u(t)) 
\end{equation} 
with 
\begin{equation}
\label{eq:initial} 
u(0) = u_0. 
\end{equation} 
We assume the following condition on $F$: 
there exist numbers $R>0$, $\rho_0>\epsilon$ and $C>0$ 
such that the following holds 
\begin{itemize} 
\item[(a)] for every $\epsilon\le \rho'<\rho\le \rho_0$, 
$F$ is a holomorphic mapping 
\[
F \colon \{u \in B_\rho: \|u-u_0\|_\rho < R\}\longrightarrow 
B_{\rho'}; 
\]

\item[(b)] for every $\epsilon\le \rho' <\rho \le \rho_0$, 
$u, v\in B_\rho$ with $\|u-u_0\|_\rho<R$, $\|v-u_0\|_\rho<R$, 
we have 
\[
\| F(u) - F(v) \|_{\rho'} \le \frac{C}{\rho-\rho'} \|u - v\|_{\rho}. 
\]

\item[(c)] for every $\epsilon \le \rho<\rho_0$, we have 
\[
\| F(u_0) \|_{\rho} \le \frac{C}{\rho_0 - \rho}. 
\]
\end{itemize}
Note that the constant $C$ is independent of $\rho, \rho', u, v$. 
\begin{theorem}[Nishida {\cite{Nishida}}]
\label{thm:Nishida} 
Under the above assumption, there exists a constant $a>0$ 
such that for every $\epsilon< \rho< \rho_0$, 
there exists a unique holomorphic solution 
\[
u\colon \{t: |t|<a(\rho-\rho_0)\} \to \{ x \in B_\rho : \| x- u_0\|_{\rho} 
< R\} 
\]
to the system \eqref{eq:diffeq}, \eqref{eq:initial}. 
\end{theorem} 

We prove the convergence of $\cA_i = A_i q^{Q_i\partial_{Q_i}}$ and 
$\Omega_\alpha = (\Phi_\alpha\bullet)$ by induction 
on the number of $t$-variables. 
For a positive number $\rho>0$, $1\le i\le r$ and $0\le \beta \le N+1$, 
define $B_{\beta,i,\rho}$ to be the space of 
functions $f(q,Q_1,\dots,Q_r,t^0,\dots,t^{\beta-1})$ 
in the variables $q, Q_1,\dots,Q_r, t^0,\dots,t^{\beta-1}$ 
such that $f$ is continuous and bounded on the set 
\begin{equation} 
\label{eq:polydisc}
D_{i,\beta}(\rho) = 
\{(q,Q_1,\dots,Q_r, t^0,\dots, t^{\beta-1}) : 
|qQ_i|\le \rho, \ |Q_j|\le \rho \ (\forall j) \ , |t^\alpha|\le \rho \ 
(\forall \alpha)\} 
\end{equation} 
and holomorphic in the interior of $D_{i,\beta}(\rho)$. 
Define the norm $\|\cdot\|_{i,\rho}$ on $B_{\beta,i,\rho}$ 
as the sup norm over $D_{i,\beta}(\rho)$. 
Then $B_{\beta,i,\rho}$ is a Banach algebra. 

Set $\tB_{\beta,i,\rho} = \End(K(X)) \otimes B_{\beta,i,\rho}$. 
Choose a norm on $K(X)_\C := K(X)\otimes \C$. 
Let  $\|\cdot\|_{\rm op}$ denote the corresponding operator 
norm on $\End(K(X)_\C)$. 
We define the norm $\|\cdot\|_{i,\rho}$ on $\tB_{\beta,i,\rho}$ by 
\[
\|A\|_{i,\rho} = \sup_{(q,Q,t) \in D_{i,\beta}(\rho)} \|A(q,Q,t)\|_{\rm op} 
\]
where $D_{i,\beta}$ is given in \eqref{eq:polydisc}. 
Lemma \ref{lem:extend_domain} and Proposition \ref{prop:A_expansion} 
show that $A_i|_{t=0}$ belongs to $\tB_{0,i,\rho}$ for some $\rho$ under 
the assumption in Theorem \ref{thm:convergence}.  

\begin{claim} 
\label{cla:induction}
We assume by induction that there exist $0\le \beta \le N$ 
and $\rho>0$ such that 
$A_i|_{t^{\beta}= \cdots = t^r =0}$ 
belongs to $\tB_{\beta,i,\rho}$ for all $1\le i\le r$. 
Then $A_i|_{t^{\beta+1} =\cdots = t^r =0}$ 
belongs to $\tB_{\beta+1,i,\rho'}$ for a possibly smaller 
$\rho'>0$. 
\end{claim} 

We prove the claim by Nishida's theorem. 
For simplicity, during the proof of the claim, we denote the restrictions 
$A_i|_{t^{\beta+1} = \cdots = t^N=0}$ and 
$\Omega_\beta|_{t^{\beta+1} = \cdots =t^N=0}$ 
by $A_i$ and $\Omega_\beta$. 
Recall that the reconstruction in the $t^\beta$-direction 
is based on the differential equation: 
\begin{equation} 
\label{eq:diffeq_A} 
(1-q) \parfrac{A_i}{t^\beta} = [ A_i q^{Q_i\partial_{Q_i}}, \Omega_\beta] 
\end{equation} 
together with the condition 
\begin{equation}
\label{eq:Omega_unit} 
\Omega_\beta \Phi_0 = \Phi_\beta. 
\end{equation} 
We rewrite the differential equation \eqref{eq:diffeq_A} 
by recalling the argument in the proof of 
Lemma \ref{lem:rational_expression}. 
Set $A_{i, \rm com} := A_i|_{q=1}$. 
Since line bundles generate $K(X)$, we can find 
polynomials $F_\alpha(x_1,\dots,x_r) \in \Q[x_1,\dots,x_r]$, 
$0\le \alpha \le N$ 
such that 
$F_\alpha(P_1^{-1},\dots,P_r^{-1})$, $0\le \alpha\le N$ 
form a basis of $K(X)$. 
Then 
\[
F_\alpha(A_{1,\rm com},\dots,A_{r,\rm com}) 
\quad \alpha=0,\dots,N
\]
form a basis of $K(X)\otimes \C[\![Q,t^0,\dots,t^\beta]\!]$ 
over $\C[\![Q,t^0,\dots,t^\beta]\!]$. 
We also have $[A_{i, \rm com}, \Omega_\beta]=0$ by \eqref{eq:diffeq_A}. 
Therefore the equation \eqref{eq:Omega_unit} gives 
\[
\Omega_\beta F_\alpha(A_{1,\rm com},\dots,A_{r,\rm com}) \Phi_0 
= F_\alpha(A_{1,\rm com},\dots, A_{r,\rm com}) \Phi_\beta 
\]
which can be written as a matrix equation 
as in \eqref{eq:Omega_determined}. 
These equations determine $\Omega_\beta$ as a rational function 
of $(A_{i,\rm com})_{i=1}^r \in \End(K(X)_\C)^r$. 
We denote this rational function by 
\[
(A_{i, \rm com})_{i=1}^r 
\longmapsto \Omega_\beta(A_{*,\rm com}) = \Omega_\beta(A_{1,\rm com},
\dots, A_{r,\rm com}). 
\] 
Let $D_i$ denote the divided difference operator: 
\[
D_i = (1- q^{Q_i\partial_{Q_i}})/(1-q).  
\]
Then we can rewrite \eqref{eq:diffeq_A} as an 
ordinary differential equation 
\[
\frac{d}{dt^\beta} A_i = -A_i \left( D_i \Omega_\beta(A_{*, \rm com}) \right) - 
\left[ \tfrac{A_{i,\rm com}-A_i}{1-q}, \Omega_\beta(A_{*, \rm com}) 
\right] 
\]
where we regard $(A_i)_{i=1}^r$ as a function of $t^\beta$ 
with values in the Banach space 
$\bigoplus_{i=1}^r \tB_{\beta,i,\rho}$. 
We write $A^0_i = A_i|_{t^\beta =0}$ for the initial value 
of $A_i$. 
It now suffices to show that the map 
\begin{equation}
\label{eq:F_A}
F(A_1,\dots,A_r) = 
\left( 
-A_i \left( D_i \Omega_\beta(A_{*,\rm com}) \right) - 
\left[ \tfrac{A_{i,\rm com}-A_i}{1-q}, \Omega_\beta(A_{*,\rm com}) 
\right] 
\right)_{i=1}^r 
\end{equation} 
satisfies the assumptions in the abstract Cauchy-Kowalevski theorem 
(Theorem \ref{thm:Nishida}). 
We choose necessary constants as follows: 
\begin{itemize} 
\item Let $R>0$ be a positive number such that the 
rational function $\Omega_\beta \colon 
\End(K(X)_\C)^r \to \End(K(X)_\C)$ 
is regular on the ball $\{(M_i)_{i=1}^r: \|M_i - P_i^{-1}\|_{\rm op} < 3R\}$. 
\item Choose $\rho_0>0$ such that for each $1\le i\le r$, 
$A_i^0 = A_i|_{t^\beta=0}$ belongs to $\tB_{\beta,i,\rho_0}$ and that 
$\|A^0_{i,\rm com} - P^{-1}_i\|_{\rm op}<R$ 
holds on the set  $\{(Q_1,\dots,Q_r, t^0,\dots,t^{\beta-1}): 
|Q_j|\le \rho_0 \ (\forall j), |t^\alpha|\le \rho_0\ (\forall \alpha) \}$. 
Here $A^0_{i,\rm com} = A_i|_{q=1, t^\beta =0}$. 
\item Choose $0<\epsilon<\rho_0$.
\end{itemize}  
We work with the family of Banach spaces 
$\{\bigoplus_{i=1}^r \tB_{\beta,i,\rho}\}_{\epsilon\le \rho 
\le \rho_0}$. 
We also define $\tB_{\beta,\rho}$ to be the space 
of $\End(K(X)_\C)$-valued functions 
$M(Q_1,\dots,Q_r, t^0,\dots,t^{\beta-1})$ 
in the variables $Q_1,\dots,Q_r, t^0,\dots,t^{\beta-1}$ 
which are continuous and bounded on the set 
\[
\{(Q_1,\dots, Q_r, t^0,\dots,t^{\beta-1}) 
: |Q_j|\le \rho\ (\forall j), |t^\alpha|\le \rho \ (\forall \alpha)\} 
\]
and holomorphic in its interior.  
The norm $\|\cdot\|_{\rho}$ on $\tB_{\beta,\rho}$ is defined 
as the sup norm over the above set. 
For $M\in \tB_{\beta,i,\rho}$, we write 
$M_{\rm com} = M|_{q=1} \in \tB_{\beta,\rho}$ 
for the restriction to $q=1$. 

\begin{sublemma} 
There exists a constant $C>0$ such that the following hold: 
\begin{itemize} 
\item[(i)] for every $\epsilon \le \rho \le \rho_0$ 
and for all $(M_i)_{i=1}^r \in \bigoplus_{i=1}^r \tB_{\beta,i,\rho}$ 
with $\|M_i - A_i^0\|_{i,\rho}<R$, 
we have 
\[
\|\Omega_\beta(M_{*,\rm com}) \|_{\rho} \le C; 
\]

\item[(ii)] for every $\epsilon \le \rho \le \rho_0$ and 
for all $(L_i)_{i=1}^r, (M_i)_{i=1}^r \in \bigoplus_{i=1}^r 
\tB_{\beta,i,\rho}$ with 
$\|L_i- A_i^0\|_{i,\rho}<R$ 
and $\|M_i - A_i^0\|_{i,\rho} <R$ for 
all $1\le i\le r$, we have 
\[
\|\Omega_\beta(L_{*, \rm com}) - \Omega_\beta(M_{*, \rm com})\|_{\rho} 
\le C \sum_{i=1}^r \|L_i - M_i\|_{i,\rho}; 
\]

\item[(iii)] for every pair of numbers 
$\epsilon \le \rho' < \rho \le \rho_0$, 
$1\le i\le r$ 
and for all $M \in \tB_{\beta,\rho}$,  
we have 
\[
\|D_i M\|_{i, \rho'} \le \frac{C}{\rho-\rho'} \|M\|_{\rho}; 
\]

\item[(iv)]  for every pair of numbers 
$\epsilon \le \rho' < \rho \le \rho_0$ 
and for all $M \in \tB_{\beta,i,\rho}$, we have 
\[
\left\|\tfrac{M_{\rm com} - M}{1-q} 
\right \|_{i, \rho'} \le \frac{C}{\rho-\rho'} \|M\|_{i,\rho}. 
\]
\end{itemize}
\end{sublemma} 
\begin{proof} 
Part (i) and (ii) follows from the fact that $\Omega_\beta(M_*)$ is 
Lipschitz continuous and bounded on the ball $\{(M_i)_{i=1}^r\in \End(K(X)_\C)^r: 
\|M_i -  P_i^{-1}\|_{\rm op}\le 2 R\}$ 
and that we are working with the sup norm. 

Next we show Part (iii). 
Note that $D_i M= (M(Q,t) - M(q^{e_i}Q,t))/(1-q)$ is defined in the region  
$\{(q,Q_1,\dots,Q_r,t^0,\dots,t^{\beta-1}) : |qQ_i|\le \rho, |Q_j|\le \rho \ (\forall j), 
|t^\alpha|\le \rho \ (\forall \alpha) \}$.  
By the mean value theorem  we have 
\begin{align*} 
\|(D_i M)(q,Q,t)\|_{\rm op} 
& = \left\|\frac{M(Q,t) - M(q^{e_i}Q,t)}{1-q} \right\|_{\rm op} 
= \left\| \frac{1}{1-q} 
\int_{qQ_i}^{Q_i} \parfrac{M}{Q_i}(Q,t) dQ_i\right\|_{\rm op} \\ 
& \le \frac{1}{|1-q|}  
\int_{qQ_i}^{Q_i} \left\| \parfrac{M}{Q_i}(Q,t) \right\|_{\rm op} |dQ_i| 
= |Q_i| \cdot \left \| \parfrac{M}{Q_i}(Q_1,\dots, Q_i^*, \dots,Q_r,t) \right\|_{\rm op}
\end{align*} 
for some $Q_i^*$ in the interval $[qQ_i, Q_i]$. 
On the other hand, if we have $|Q_j|\le \rho'$ ($\forall j$) 
and $|t^\alpha|\le \rho'$ ($\forall \alpha$), then 
\[
\left\| \parfrac{M}{Q_i} (Q,t)\right\|_{\rm op} 
\le \frac{1}{\rho-\rho'} \| M \|_{\rho} 
\]
by Cauchy's integral formula. Combining the two estimates, we 
obtain 
\[
\|D_i M\|_{i,\rho'} \le \rho_0 \left\|\parfrac{M}{Q_i} \right\|_{\rho'} 
\le \frac{\rho_0}{\rho-\rho'} \|M\|_{\rho} 
\]
as required. 

Part (iv) is essentially identical to Part (iii). 
By the mean value theorem we have: 
\[
\left \|\tfrac{M(1,Q,t) - M(q,Q,t)}{1-q}\right \|_{\rm op} 
\le \left \|\parfrac{M}{q}(q_*, Q,t) \right \|_{\rm op}  
\]
for some $q_*$ in the interval $[q,1]$. 
If $|qQ_i|\le \rho'$, $|Q_j|\le \rho'$ ($\forall j$) 
and $|t^\alpha|\le \rho'$ ($\forall \alpha$), 
we have 
\[
\left\|\parfrac{M}{q}(q,Q,t)\right\| \le \frac{|Q_i|}{\rho-\rho'} \|M\|_{i,\rho}
\le \frac{\rho_0}{\rho-\rho'} \|M\|_{i,\rho}  
\]
by Cauchy's integral formula. 
Therefore we have 
\[
\left\|\tfrac{M_{\rm com} - M}{1-q} \right\|_{i, \rho'}  
\le \left \| \parfrac{M}{q} \right\|_{i, \rho'}
\le \frac{\rho_0}{\rho-\rho'} \|M\|_{i,\rho} 
\]
as required. 
\end{proof} 

The above sublemma shows that $F(A_1,\dots,A_r)$ 
in \eqref{eq:F_A} gives a holomorphic map 
\[
F \colon \left\{(A_i)_{i=1}^r \in \bigoplus_{i=1}^r \tB_{\beta,i,\rho} : 
\|A_i - A_i^0\|_{i,\rho} < R \right \} 
\longrightarrow 
\bigoplus_{i=1}^r \tB_{\beta,i,\rho'}
\] 
for every pair $(\rho,\rho')$ 
such that $\epsilon \le \rho' < \rho \le \rho_0$. 
We write $F(A_1,\dots,A_r) = (F_i(A_*))_{i=1}^r$. 
Moreover, for all $(L_i)_{i=1}^r, (M_i)_{i=1}^r 
\in \bigoplus_{i=1}^r \tB_{\beta,i,\rho}$ with $\|L_i - A^0_i \|_{i,\rho}<R$ 
and $\|M_i - A^0_i \|_{i,\rho} < R$, and for all 
$\epsilon\le \rho'<\rho\le \rho_0$, 
we have 
\begin{align*} 
& \|F_i(L_*)  - F_i(M_*)\|_{i,\rho'} \le 
\left\| (M_i - L_i) D_i\Omega_\beta(M_{*,\rm com}) 
+ L_i (D_i \Omega_\beta(M_{*,\rm com})- 
D_i\Omega_\beta (L_{*,\rm com})) \right\|_{i, \rho'}  \\
& \qquad \quad + \left \| 
\left[ \tfrac{M_{i,\rm com} -M_i}{1-q}, \Omega_\beta(M_{*,\rm com})- 
\Omega_\beta (L_{*, \rm com}) \right] 
+ 
\left[ \tfrac{(M_i-L_i)_{\rm com} -(M_i-L_i)}{1-q}, 
\Omega_\beta (L_{*, \rm com}) \right] 
\right\|_{i, \rho'} \\ 
&\le \|M_i -L_i\|_{i,\rho'} \frac{C}{\rho-\rho'} 
\|\Omega_\beta(M_{*,\rm com})\|_{\rho} 
+ \|L_i \|_{i,\rho'} \frac{C}{\rho-\rho'} 
\|\Omega_\beta (M_{*,\rm com}) - \Omega_\beta (L_{*,\rm com})
\|_{\rho} \\ 
& 
\quad + 2 \frac{C}{\rho-\rho'} \|M_i \|_{i, \rho}  
\| \Omega_\beta(M_{*, \rm com}) - \Omega_\beta(L_{*, \rm com})\|_{\rho'} 
+ 2 \frac{C}{\rho-\rho'} \|M_i-L_i\|_{i, \rho} 
\|\Omega_\beta(L_{*,\rm com})\|_{\rho'} \\ 
& \le \left(3 C^2 
+ 3 C^2 (R+ \|A^0_i\|_{i,\rho_0} ) 
\right)
\frac{1}{\rho -\rho'} \sum_{i=1}^r \|L_i-M_i\|_{i,\rho} 
\end{align*} 
We also have for $\epsilon \le \rho<\rho_0$, 
\begin{align*} 
\|F_i(A_*^0) \|_{i,\rho} 
& \le \|A_i^0 (D_i \Omega_\beta(A_*^0))\|_{i,\rho} 
+\left \|\left[\tfrac{A_{i,\rm com}^0 - A_i^0}{1-q}, \Omega_\beta(A_*^0)
\right] \right \|_{i,\rho}  \\ 
& \le \|A_i^0\|_{i,\rho} \frac{C}{\rho_0-\rho} 
\|\Omega_\beta(A_*^0)\|_{\rho_0} 
+ \frac{2C}{\rho_0-\rho} \|A_i^0\|_{i,\rho_0} 
\|\Omega_\beta(A_*^0)\|_{\rho} \\ 
& \le 3 C^2 \|A_i^0\|_{i,\rho_0} 
\frac{1}{\rho_0- \rho}. 
\end{align*} 
The proof of Claim \ref{cla:induction} is complete. 

By induction we now know that there exists a positive number 
$\rho>0$ such that, for all $1\le i\le r$, $A_i$ is analytic on the region 
\[
\{(q,Q_1,\dots,Q_r, t^0,\dots,t^N) : |qQ_i|< \rho, 
\ |Q_j|< \rho \ (\forall j), |t^\alpha| < \rho \ 
(\forall \alpha) \}. 
\]
As discussed above, the big quantum products 
$\Omega_\beta = (\Phi_\beta \bullet)$ are determined 
as rational functions of $A_{i,\rm com}$, $1\le i\le r$, 
and hence, by taking a smaller $\rho$ if necessary, 
are analytic on the region: 
\begin{equation} 
\label{eq:rho_ball}
\{(Q_1,\dots,Q_r, t^0, \dots, t^N): 
|Q_j|<\rho, \ (\forall j), |t^\alpha|<\rho \ 
(\forall \alpha)\}. 
\end{equation} 
By Proposition \ref{prop:convergence_S}, by taking 
a smaller $\rho$ if necessary, the fundamental solutions 
$S$ and $T$ are convergent on 
\[
\{(q, Q_1,\dots,Q_r, t^0, \dots, t^N): 
|q| \neq 1, |Q_j|<\rho \ (\forall j), |t^\alpha|<\rho \ 
(\forall \alpha) \}. 
\]
By Proposition \ref{prop:unitarity}, we know that the metric 
tensor $G_{\alpha\beta} = 
G(\Phi_\alpha,\Phi_\beta) = g(\ov{T} \Phi_\alpha, \Phi_\beta)$ 
is analytic on the region \eqref{eq:rho_ball}. 
By the string equation, the genus zero potential $F(t)$ satisfies 
\[
F(t) = F_{000}(t) - \chi(\cO) - \chi(t) - \frac{1}{2} \chi(t\otimes t). 
\]
Because $F_{000}(t) = G_{00}(t)$, it follows that 
$F(t)$ is also analytic on the region \eqref{eq:rho_ball}. 
The proof of Theorem \ref{thm:convergence} is now complete. 

\subsection{Applications: semisimplicity of $\C\Proj^N$ and $\Fl_3$}

\begin{proposition} 
\label{prop:projective} 
The genus-zero quantum $K$-theory of 
$\C\Proj^N$ is analytic, i.e.~the conclusions of 
Theorem \ref{thm:convergence} hold for $\C\Proj^N$.  
Moreover, the big quantum $K$-rings of $\C\Proj^N$  
is semisimple for a generic $(Q,t)$. 
\end{proposition} 
\begin{proof} 
By Proposition \ref{prop:A_projective}, 
the $q$-shift operator $\cA|_{t=0}$ of the projective space 
is analytic in $q$ and $Q$. 
Therefore the conclusions of Theorem \ref{thm:convergence} 
hold for $\C\Proj^N$. 
The computation in Proposition \ref{prop:A_projective} 
also shows that $A_{\rm com}|_{t=0}$ with $Q\neq 0$
is a semisimple endomorphism with pairwise distinct eigenvalues. 
(The difference equation \eqref{fdj} implies 
$(1-A_{\rm com})^{N+1} =Q$ and the eigenvalues of 
$1-A_{\rm com}$ are $(N+1)$th roots of $Q$.)
Because $A_{\rm com}$ commutes with the quantum product 
(Corollary \ref{cor:Acom}), 
this implies that the quantum $K$-ring is semisimple for $t=0$, $Q\neq 0$. 
Being an open property, the semisimplicity holds for a generic 
$(Q,t)$. 
\end{proof} 

\begin{proposition} 
\label{prop:flag3} 
The same conclusion as Proposition \ref{prop:projective} 
holds for $\Fl_3$.  
\end{proposition} 
\begin{proof} 
The proof is completely parallel to Proposition \ref{prop:projective}. 
It suffices to use the computation in \S \ref{subsec:flag} 
instead of Proposition \ref{prop:A_projective}. 
\end{proof} 

\begin{remark} 
The Hirzebruch-Riemann-Roch theorem of Givental-Tonita \cite{Givental-Tonita} 
says that the localization at $q=1$ of the 
Lagrangian cone of the true quantum $K$-theory 
coincides with the Lagrangian cone of the fake quantum $K$-theory. 
Because the Lagrangian cone of the fake quantum $K$-theory 
is a symplectic transform of the Lagrangian cone 
of the quantum cohomology, it follows that the $F$-manifolds 
associated to the quantum cohomology and 
the quantum $K$-theory are isomorphic under a formal change of co-ordinates. 
(Note that the manifold of tangent spaces 
to an overruled Lagrangian cone is equipped with 
the structure of an $F$-manifold, independent of the choice of a 
polarization.) 
Therefore, under convergence assumption, we should expect that 
the quantum $K$-ring is generically semisimple if and only if the 
quantum cohomology is generically semisimple. 
\end{remark} 

\appendix 
\raggedbottom
\section{Table of $\C\Proj^2$ Invariants}

We record some $K$-theoretic GW invariants of $\C\Proj^2$
for degree one, two and three. 
The $(i+1,j+1)$ entry of the matrices below gives the correlator 
\[
\left\langle
\overbrace{[H],\dots, [H]}^{i}, 
\overbrace{[{\rm pt}],\dots, [{\rm pt}]}^{j}
\right\rangle_{0,i+j,d}^{\C\Proj^2} 
\]
where $[H] = 1 - [\cO(-1)]$ is the class of a hyperplane 
and $[{\rm pt}] = (1 - [\cO(-1)])^2$ is the class of a point. 

\vspace{10pt}

\underline{Degree-one invariants}: 
\scriptsize 
\[ \left[ \begin {array}{ccccccccccccccc} 1&1&1&0&0&0&0&0&0&0&0&0&0&0&0
\\\noalign{\medskip}1&1&1&0&0&0&0&0&0&0&0&0&0&0&0
\\\noalign{\medskip}1
&1&1&0&0&0&0&0&0&0&0&0&0&0&0
\\\noalign{\medskip}1&1&1&0&0&0&0&0&0&0&0&0
&0&0&0\\\noalign{\medskip}1&1&1&0&0&0&0&0&0&0&0&0&0&0&0
\\\noalign{\medskip}1&1&1&0&0&0&0&0&0&0&0&0&0&0&0
\\\noalign{\medskip}1
&1&1&0&0&0&0&0&0&0&0&0&0&0&0
\\\noalign{\medskip}1&1&1&0&0&0&0&0&0&0&0&0
&0&0&0\\\noalign{\medskip}1&1&1&0&0&0&0&0&0&0&0&0&0&0&0
\\\noalign{\medskip}1&1&1&0&0&0&0&0&0&0&0&0&0&0&0
\\\noalign{\medskip}1
&1&1&0&0&0&0&0&0&0&0&0&0&0&0
\\\noalign{\medskip}1&1&1&0&0&0&0&0&0&0&0&0
&0&0&0\\\noalign{\medskip}1&1&1&0&0&0&0&0&0&0&0&0&0&0&0
\\\noalign{\medskip}1&1&1&0&0&0&0&0&0&0&0&0&0&0&0
\\\noalign{\medskip}1
&1&1&0&0&0&0&0&0&0&0&0&0&0&0\end {array} 
\right] 
\]
\normalsize 

\vspace{10pt}

\underline{Degree-two invariants}: 
\scriptsize 
\[
\left[ \begin {array}{cccccccccccccccccccc} 1&1&1&1&1&1&0&0&0&0&0&0&0
&0&0&0&0&0&0&0
\\\noalign{\medskip}1&1&1&1&1&2&0&0&0&0&0&0&0&0&0&0&0&0&0
&0\\\noalign{\medskip}1&1&1&1&0&4&0&0&0&0&0&0&0&0&0&0&0&0&0&0
\\\noalign{\medskip}1&1&1&2&-4&8&0&0&0&0&0&0&0&0&0&0&0&0&0&0
\\\noalign{\medskip}1&1&1&8&-16&16&0&0&0&0&0&0&0&0&0&0&0&0&0&0
\\\noalign{\medskip}1&1&0&32&-48&32&0&0&0&0&0&0&0&0&0&0&0&0&0&0
\\\noalign{\medskip}1&1&-8&112&-128&64&0&0&0&0&0&0&0&0&0&0&0&0&0&0
\\\noalign{\medskip}1&1&-48&352&-320&128&0&0&0&0&0&0&0&0&0&0&0&0&0&0
\\\noalign{\medskip}1&2&-208&1024&-768&256&0&0&0&0&0&0&0&0&0&0&0&0&0&0
\\\noalign{\medskip}0&12&-768&2816&-1792&512&0&0&0&0&0&0&0&0&0&0&0&0&0
&0\\\noalign{\medskip}-9&72&-2560&7424&-4096&1024&0&0&0&0&0&0&0&0&0&0&0
&0&0&0\\\noalign{\medskip}-60&352&-7936&18944&-9216&2048&0&0&0&0&0&0&0
&0&0&0&0&0&0&0\\\noalign{\medskip}-292&1472&-23296&47104&-20480&4096&0
&0&0&0&0&0&0&0&0&0&0&0&0&0\\\noalign{\medskip}-1216&5504&-65536&114688
&-45056&8192&0&0&0&0&0&0&0&0&0&0&0&0&0&0\\\noalign{\medskip}-4576&
18944&-178176&274432&-98304&16384&0&0&0&0&0&0&0&0&0&0&0&0&0&0
\\\noalign{\medskip}-16000&61184&-471040&647168&-212992&32768&0&0&0&0&0
&0&0&0&0&0&0&0&0&0\\\noalign{\medskip}-52864&187904&-1216512&1507328&-
458752&65536&0&0&0&0&0&0&0&0&0&0&0&0&0&0\\\noalign{\medskip}-166912&
553984&-3080192&3473408&-983040&131072&0&0&0&0&0&0&0&0&0&0&0&0&0&0
\\\noalign{\medskip}-507648&1579008&-7667712&7929856&-2097152&262144&0
&0&0&0&0&0&0&0&0&0&0&0&0&0\\\noalign{\medskip}-1496064&4374528&-
18808832&17956864&-4456448&524288&0&0&0&0&0&0&0&0&0&0&0&0&0&0
\end {array} \right] 
\]
\normalsize 

\vspace{10pt} 

\underline{Degree-three invariants}: 
\tiny 
\[
 \left[ \begin {array}{ccccccccccccc} 1&1&1&1&1&1&1&-2&12&0&0&0&0
\\\noalign{\medskip}1&1&1&1&1&1&4&-24&36&0&0&0&0
\\\noalign{\medskip}1&1&1&1&1&0&40&-126&108&0&0&0&0
\\\noalign{\medskip}1&1&1&1&1&-30&279&-540&324&0&0&0&0
\\\noalign{\medskip}1&1&1&1&15&-333&1539&-2106&972&0&0&0&0
\\\noalign{\medskip}1&1&1&1&243&-2457&7398&-7776&2916&0&0&0&0
\\\noalign{\medskip}1&1&1&-42&2403&-14742&32562&-27702&8748&0&0&0&0
\\\noalign{\medskip}1&1&1&-819&18063&-77760&134865&-96228&26244&0&0&0&0
\\\noalign{\medskip}1&1&70&-9003&114453&-375435&534357&-328050&
78732&0&0&0&0
\\\noalign{\medskip}1&1&1525&-74643&645165&-1699299&
2047032&-1102248&236196&0&0&0&0
\\\noalign{\medskip}1&-74&19039&-
518661&3339549&-7322076&7637004&-3661038&708588&0&0&0&0
\\\noalign{\medskip}1&-1913&177040&-3187587&16199109&-30351186&
27890811&-12045996&2125764&0&0&0&0
\\\noalign{\medskip}
133&-27116&
1364272&-17889387&74657619&-121936185&100088055
&-39326634&6377292&0&0&0&0
\\\noalign{\medskip}3433&-281948&9211867&-93599739&330142959&-
477411165&353939706&-127545840&19131876&0&0&0&0
\\\noalign{\medskip}50155&-2402093&56356033&-463118232&1410975855&-
1829219922&1236131766&-411335334&57395628&0&0&0&0
\end {array}
 \right] 
\]
\normalsize 
Notice that the $K$-theoretic GW invariants 
$1 = \corr{[\rm pt],[\rm pt]}_{0,2,1}$, 
$1 = \corr{[\rm pt],\dots,[\rm pt]}_{0,5,2}$, 
$12 = \corr{[\rm pt],\dots,[\rm pt]}_{0,8,3}$ 
coincide with the corresponding cohomological GW invariants 
(the number of degree $d$ rational curves passing through 
$3d-1$ points in $\C\Proj^2$). 

\vspace{10pt} 
\noindent 
{\bf Acknowledgments.} 
We thank Changzheng Li and Leonardo C. Mihalcea for directing us 
to references on the finiteness results. We also thank 
Toshiaki Maeno for a very helpful discussion on the small quantum 
$K$-ring of flag manifolds. 
H.I.~thanks Professor Hiraku Nakajima for organizing
a postdoc seminar in 2005, where he learned quantum $K$-theory. 
H.I. is supported by JSPS Grant-in-Aid for Scientific Research (C)
25400069, T.M. is supported by JSPS Grant-in-Aid, and T.M. and
V.T. acknowledge the World Premiere International Research Center
Initiative (WPI Initiative), Mext, Japan.

\end{document}